\documentclass[12pt,english,american]{amsart}
\usepackage[T1]{fontenc}
\usepackage[latin9]{inputenc}
\usepackage[letterpaper]{geometry}
\usepackage{babel}
\usepackage{float}
\usepackage{amsthm}
\usepackage{amstext}
\usepackage{esint}
\usepackage[unicode=true,pdfusetitle,
 bookmarks=true,bookmarksnumbered=false,bookmarksopen=false,
 breaklinks=false,pdfborder={0 0 1},backref=false,colorlinks=false]
 {hyperref}

\makeatletter

\providecommand{\tabularnewline}{\\}
\floatstyle{ruled}
\newfloat{algorithm}{tbp}{loa}
\providecommand{\algorithmname}{Algorithm}
\floatname{algorithm}{\protect\algorithmname}

\numberwithin{equation}{section}
\numberwithin{figure}{section}
\numberwithin{table}{section}
\theoremstyle{plain}
\newtheorem{thm}{\protect\theoremname}[section]
  \theoremstyle{plain}
  \newtheorem{prop}[thm]{\protect\propositionname}
  \theoremstyle{definition}
  \newtheorem{example}[thm]{\protect\examplename}
  \theoremstyle{plain}
  \newtheorem{cor}[thm]{\protect\corollaryname}
  \theoremstyle{definition}
  \newtheorem{defn}[thm]{\protect\definitionname}
  \theoremstyle{plain}
  \newtheorem{lem}[thm]{\protect\lemmaname}

\usepackage[arrow,matrix,curve,cmtip,ps]{xy}
\usepackage{tikz}
\usepackage{pgflibraryarrows}
\usetikzlibrary{calc}
\definecolor{ffqqww}{rgb}{1,0,0.4}
\usepackage{stmaryrd}

\@ifundefined{showcaptionsetup}{}{%
 \PassOptionsToPackage{caption=false}{subfig}}
\usepackage{subfig}
\makeatother

  \addto\captionsamerican{\renewcommand{\corollaryname}{Corollary}}
  \addto\captionsamerican{\renewcommand{\definitionname}{Definition}}
  \addto\captionsamerican{\renewcommand{\examplename}{Example}}
  \addto\captionsamerican{\renewcommand{\lemmaname}{Lemma}}
  \addto\captionsamerican{\renewcommand{\propositionname}{Proposition}}
  \addto\captionsamerican{\renewcommand{\theoremname}{Theorem}}
  \addto\captionsenglish{\renewcommand{\corollaryname}{Corollary}}
  \addto\captionsenglish{\renewcommand{\definitionname}{Definition}}
  \addto\captionsenglish{\renewcommand{\examplename}{Example}}
  \addto\captionsenglish{\renewcommand{\lemmaname}{Lemma}}
  \addto\captionsenglish{\renewcommand{\propositionname}{Proposition}}
  \addto\captionsenglish{\renewcommand{\theoremname}{Theorem}}
  \providecommand{\corollaryname}{Corollary}
  \providecommand{\definitionname}{Definition}
  \providecommand{\examplename}{Example}
  \providecommand{\lemmaname}{Lemma}
  \providecommand{\propositionname}{Proposition}
\addto\captionsenglish{\renewcommand{\algorithmname}{Algorithm}}
\providecommand{\theoremname}{Theorem}

\begin{document}

\title{Biased Weak Polyform Achievement Games}

\author{Ian Norris}

\author{N\'andor Sieben}

\address{Northern Arizona University, Department of Mathematics and Statistics,
Flagstaff AZ 86011-5717, USA}

\email{ian52n@gmail.com}

\email{nandor.sieben@nau.edu}



\date{\the\month/\the\day/\the\year}
\begin{abstract}
In a biased weak $(a,b)$ polyform achievement game, the maker and
the breaker alternately mark $a,b$ previously unmarked cells on an
infinite board, respectively. The maker's goal is to mark a set of
cells congruent to a polyform. The breaker tries to prevent the maker
from achieving this goal. A winning maker strategy for the $(a,b)$
game can be built from winning strategies for games involving fewer
marks for the maker and the breaker. A new type of breaker strategy
called the priority strategy is introduced. The winners are determined
for all $(a,b)$ pairs for polyiamonds and polyominoes up to size
four.
\end{abstract}
\maketitle
\newcommand\trrad{.3} 

\newcommand\trGoU[2]{
  (90:#2*\trrad) ++(30:#2*\trrad) ++(30:#1*\trrad) ++(-30:#1*\trrad)
}

\newcommand\trGoD[2]{
  (90:#2*\trrad) ++(30:#2*\trrad) ++(30:#1*\trrad) ++(-30:#1*\trrad) ++(90:\trrad) ++(-30:\trrad)
}

\newcommand\trUB[2]{  
  \draw \trGoU{#1}{#2}
     +(90:\trrad) -- +(210:\trrad) -- +(330:\trrad) -- cycle 
     +(90:.7*\trrad) -- +(210:.7*\trrad) -- +(330:.7*\trrad) -- cycle;
}

\newcommand\trUBC[3]{
\trUB{#1}{#2}
\path \trGoU{#1}{#2} coordinate (P0);
\node at (P0) {$\scriptstyle #3 $};
}

\newcommand\trUM[2]{  
  \draw \trGoU{#1}{#2} coordinate (P0)
     +(90:\trrad) -- +(210:\trrad) -- +(330:\trrad) -- cycle; 
  \draw[fill,color=black!80] (P0) +(90:.7*\trrad) -- +(210:.7*\trrad) -- +(330:.7*\trrad) -- cycle;
}

\newcommand\trUMC[3]{
\trUM{#1}{#2}
\path \trGoU{#1}{#2} coordinate (P0);
\node[color=white] at (P0) {$\scriptstyle #3 $};
}

\newcommand\trDB[2]{  
  \draw \trGoD{#1}{#2}
     +(-90:\trrad) -- +(-210:\trrad) -- +(-330:\trrad) -- cycle 
     +(-90:.7*\trrad) -- +(-210:.7*\trrad) -- +(-330:.7*\trrad) -- cycle;
}

\newcommand\trDBC[3]{
\trDB{#1}{#2}
\path \trGoD{#1}{#2} coordinate (P0);
\node at (P0) {$\scriptstyle #3 $};
}

\newcommand\trDM[2]{  
  \draw \trGoD{#1}{#2} coordinate (P0)
     +(-90:\trrad) -- +(-210:\trrad) -- +(-330:\trrad) -- cycle; 
  \draw[fill,color=black!80] (P0) +(-90:.7*\trrad) -- +(-210:.7*\trrad) -- +(-330:.7*\trrad) -- cycle;
}

\newcommand\trU[6]{  
  \draw \trGoU{#1}{#2} coordinate (P0) 
    +(90:\trrad) -- +(210:\trrad) -- +(330:\trrad) -- cycle;
  \node at (P0) {$\scriptstyle #3$};
  \node at ($ (P0)+(90:0.5*\trrad) $) {$\scriptstyle #4$};
  \node at ($ (P0)+(210:0.5*\trrad) $) {$\scriptstyle #5$};
  \node at ($ (P0)+(-30:0.5*\trrad) $) {$\scriptstyle #6$};
}

\newcommand\trD[6]{  
  \draw \trGoD{#1}{#2} coordinate (P0)
    +(-90:\trrad) -- +(30:\trrad) -- +(150:\trrad) -- cycle;
  \node at (P0) {$\scriptstyle #3$};
  \node at ($ (P0)+(-90:0.5*\trrad) $) {$\scriptstyle #4$};
  \node at ($ (P0)+(150:0.5*\trrad) $) {$\scriptstyle #5$};
  \node at ($ (P0)+(30:0.5*\trrad) $) {$\scriptstyle #6$};
}

\newcommand\trDMC[3]{
\trDM{#1}{#2}
\path \trGoD{#1}{#2} coordinate (P0);
\node[color=white] at (P0) {$\scriptstyle #3 $};
}

\newcommand\trTileUD[4]{
\trU{#1}{#2}{}{}{}{}
\trD{#3}{#4}{}{}{}{}
\path \trGoU{#1}{#2} coordinate (A0);
\path \trGoD{#3}{#4} coordinate (A1);
\draw[line width=1.6pt,color=ffqqww] (A0) -- (A1);
}


\def\a{\text{$\shortrightarrow$}}

\section{Introduction}

A \emph{plane polyform} is a figure constructed by joining finitely
many congruent basic polygons along their edges. If the basic polygons
are \emph{cells} of a regular tiling of the plane by squares, equilateral
triangles or regular hexagons, then the polyform is called a \emph{polyomino},
\emph{polyiamond} or \emph{polyhex} respectively. If the cells come
from a regular tiling of the space by cubes, then the polyform is
called a polycube. An \emph{animal} is a polyomino, polyiamond, polyhex
or polycube. We only consider animals up to congruence, that is, rotations
and reflections of an animal are considered to be the same. The number
of cells $s(A)$ of an animal $A$ is called the \emph{size} of $A$.
The standard reference for polyominoes is \cite{golomb:polyominoes}. 

In a \emph{weak animal $(a,b)$ achievement game} two players alternately
mark $a$ and $b$ previously unmarked cells respectively using their
own colors. The first player (the \emph{maker}) tries to mark a copy
of a given goal animal. The second player (the \emph{breaker}) tries
to prevent the maker from achieving his goal. An animal is an $(a,b)$-\emph{winner}
if the maker can win the $(a,b)$ achievement game. Otherwise the
animal is called a \emph{loser}. Achievement games are studied, for
example, in \cite{beck:book,ww3,bode.harborth:hexagonal,bode.harborth:triangle,gardner.martin:mathematical,harary:achievement*1}.

In Section~\ref{sec:Preliminaries} we describe the pairing strategies
and proof sequences which are the standard descriptions of breaker
and maker strategies. We also prove some basic results.

Biased games \cite[Sections 30--33]{beck:book} are more complex than
the regular $(1,1)$ game. It many cases, it is possible to decompose
a biased game into simpler biased games involving fewer marks in each
turn. In Section~\ref{sec:TheACBgame}, we describe how an $(a,b)$
maker strategy can be built from maker strategies for simpler games.

The most important strategy for the breaker for an unbiased game is
the pairing strategy. In fact, a long-standing difficulty is that
the pairing strategy is almost the only tool we have for the breaker.
The pairing strategy generalizes for $(1,b)$ games but the generalization
does not seem straightforward for $(a,b)$ games with $a\ge2$. We
remedy this problem in Section~\ref{sec:The-priority-strategy} with
the introduction of the \emph{priority strategy}. 

Given an animal, our goal is to determine all the $(a,b)$ pairs for
which the animal is a winner. This information is collected in the
\emph{threshold sequence} for the animal described in Section~\ref{sec:The-threshold-sequence}.
In Sections~\ref{sec:Polyiamonds} and \ref{sec:Polyominoes}, we
find the threshold sequence for each polyiamond and polyomino of size
smaller than 5. One of the polyominoes requires a more sophisticated
version of the priority strategy called \emph{history dependent priority
strategy}. In Section~\ref{sec:historyDep}, we describe this strategy
and we present an algorithm for verifying that a history dependent
priority strategy works. The paper ends with an unsolved problems
section.

The authors thank Ian Douglas and Steve Wilson for helpful discussions
about the material.

\section{\label{sec:Preliminaries}Preliminaries}

A strategy for the maker can be captured by a \emph{proof sequence}
$(s_{0},\ldots,s_{n})$ of situations \cite{bode.harborth:hexagonal,harborth.seemann:handicap,prooftree,sieben.deabay:polyomino}.
A situation $s_{i}=(C_{s_{i}},N_{s_{i}})$ is an ordered pair of disjoint
sets of cells. We think of the core $C_{s_{i}}$ as a set of cells
marked by the maker and the neighborhood $N_{s_{i}}$ as a set of
cells not marked by the breaker. A situation is the not necessarily
connected part of the playing board that is important for the maker.
A situation does not contain any of the breaker's marks. Those marks
are not important as long as the situation contains enough empty cells
in the neighborhood. Just like polyominoes, congruent situations are
considered to be the same. In the situations of a proof sequence,
it is always the breaker who is about to mark a cell. The game progresses
from $s_{n}$ towards $s_{0}$. We require that $C_{s_{0}}$ is the
goal polyomino and $N_{s_{0}}=\emptyset$. This means that the maker
already won by marking the cells in $C_{s_{0}}$ and there is no need
for any cells on the board in $N_{s_{0}}$. For each $i\in\{1,\ldots,n\}$
we also require that if the breaker marks $b$ or fewer cells in $N_{s_{i}}$,
then the maker can mark $a$ cells of $N_{s_{i}}$ not marked by the
breaker and reach a position $s_{j}$ closer to his goal, that is,
satisfying $j<i$. More precisely, for all $\{x_{1},\ldots,x_{b}\}\subseteq N_{s_{i}}$
there must be an $\{\tilde{x}_{1},\ldots,\tilde{x}_{a}\}\in N_{s_{i}}\setminus\{x_{1},\ldots,x_{b}\}$
and a $j\in\{0,\ldots,i-1\}$ such that
\[
C_{s_{j}}\subseteq C_{s_{i}}\cup\{\tilde{x}_{1},\ldots,\tilde{x}_{a}\}\text{ and }N_{s_{j}}\subseteq C_{s_{i}}\cup N_{s_{i}}\setminus\{x_{1},\ldots,x_{b}\}.
\]
Figures~\ref{fig:proofSeq11T32} and \ref{fig:SchProofSeqB} show
examples of proof sequences. We present proof sequences graphically.
In the figures, filled cells represent the marks of the maker. Cells
with letters in them are the neighborhood cells that must be unmarked.
Each letter represents a possible continuation for the maker. After
the marks of the breaker, the maker picks a letter unaffected by the
breaker marks. The maker marks the cells with the capital version
of this letter. The cells with the lower case version of the chosen
letter become the neighborhood cells of the new situation. Each situation
is constructed to make sure that the breaker cannot mark $b$ cells
which together contain a lower case or capital copy of all the available
letters. We include a flow chart for each proof sequence. The letter
on the arrows of the flow chart is used to determine which situation
the maker can reach by picking that letter.

Most of the known strategies for the breaker are based on pairings
of the cells of the board. A $b$-\emph{paving} of the board is a
symmetric and irreflexive relation on the set of cells where each
cell is related to at most $b$ other cells. Figure~\ref{fig:TPavings}
shows two examples of $1$-pavings and one example of a $2$-paving.
A $b$-paving determines the following \emph{paving strategy} for
the breaker in the $(1,b)$ game. In each turn, the breaker marks
the unmarked cells related to the cell last marked by the maker. If
there are fewer than $b$ such cells, then she uses her remaining
marks randomly. If the breaker follows the paving strategy, then the
maker cannot mark two related cells during a game. This allows the
breaker to win if every placement of the goal animal on the board
contains a pair of related cells. 

Two cells of a regular tiling are \emph{adjacent} if they share a
common edge. The \emph{exterior boundary $E(A)$} of the animal $A$
is the set of cells outside of $A$ but adjacent to a cell of $A$.
The \emph{site-perimeter} of $A$ is the number of cells $p(A):=|E(A)|$
in the exterior boundary (see \cite{sieben:sitePerimeter,sieben:polyominoes}).
Let $\delta$ be the size of the site-perimeter of a single cell.
Note that $\delta$ is 3, 4 and 6 on the triangular, rectangular and
hexagonal board respectively.
\begin{prop}
\label{pro:surround}Let $A$ be an animal. If $a<|A|$ and $a\delta\le b$
then $A$ is an $(a,b)$-loser.\end{prop}
\begin{proof}
Since $a<|A|$, the maker cannot build the goal animal using only
the marks of a single turn. The size of the site-perimeter of the
marks of the maker in a single turn is at most $a\delta$ so this
whole site perimeter can be marked by the breaker. If the breaker
marks every cell of the site-perimeter of the maker's mark in each
turn, then the maker cannot build the goal animal using some marks
from previous turns either.
\end{proof}
\begin{figure}
\renewcommand\trrad{.35}

\hfill{}\subfloat[\label{fig:SchProofSeqA}]{%
\begin{tabular}{ccccc}
 & $s_{0}$ &  & $s_{1}$ & \tabularnewline
 & \begin{tikzpicture}[baseline=-0.5ex] 
\draw (.4,0) ellipse (20pt and 10pt);
\node at (0.8,0) {$\scriptstyle A$};
\end{tikzpicture} &  & $a\times$\begin{tikzpicture}[baseline=-0.5ex] 
\draw (.4,0) ellipse (20pt and 10pt);
\draw[rotate=120] (.4,0) ellipse (20pt and 10pt);
\draw[rotate=-120] (.4,0) ellipse (20pt and 10pt);
\node at (0.8,0) {$\scriptstyle A_1$};
\node at (120:0.8) {$\scriptstyle A_2$};
\node at (-120:0.8) {$\scriptstyle A_k$};
\node at (0,0) {$\scriptstyle x$};
\end{tikzpicture} & \tabularnewline
\end{tabular}

}\hfill{}\subfloat[\label{fig:SchProofSeqB}]{%
\begin{tabular}{ccccc}
 & $s_{0}$ &  & $s_{1}$ & \tabularnewline
 & \begin{tikzpicture}[baseline=-0.5ex] 
\trUM{0}{0}
\trDM{0}{0}
\trUM{1}{0}
\end{tikzpicture} &  & $2\times$\begin{tikzpicture}[baseline=-0.5ex] 
\trUM{0}{0}
\trD{0}{0}{A}{}{}{}
\trU{1}{0}{A}{}{}{}
\trD{-1}{0}{B}{}{}{}
\trU{-1}{1}{B}{}{}{}
\trD{0}{-1}{C}{}{}{}
\trU{0}{-1}{C}{}{}{}
\end{tikzpicture} & \tabularnewline
\end{tabular}\lower-0.4cm\hbox{
$\xymatrix@=5mm{
s_1 \ar[d]_{a,b,c}  \\
s_0 \\
}$
}

}\hfill{}

\caption{(a) Schematic $(a,b)$ proof sequence for the animal $A$ assuming
$|A|=a-1$ and $ak\ge b$. (b) A $(2,5)$-proof sequence for the polyiamond
$T_{3,1}$.}
\end{figure}
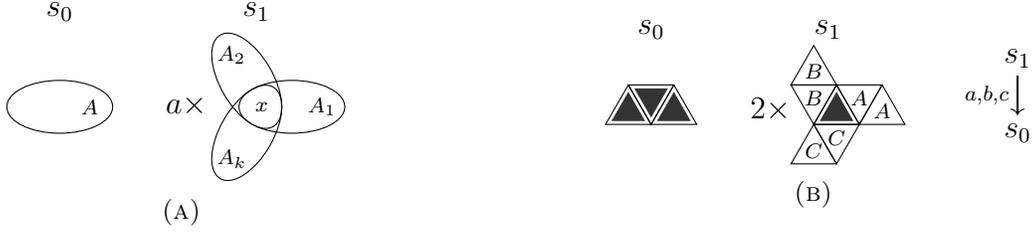

\begin{prop}
\label{pro:twostep}Let $A$ be an animal with $|A|=a+1$. Assume
$A$ has $k$ placements $A_{1},\ldots,A_{k}$ on the board such that
$A_{i}\cap A_{j}=\{x\}$ for all $i,j\in\{1,\ldots,k\}$ where $x$
is a cell. If $ak>b$ then $A$ is an $(a,b)$-winner.\end{prop}
\begin{proof}
Figure~\ref{fig:SchProofSeqA} shows a schematic proof sequence.
The maker can create situation $s_{1}$ in the first turn by marking
$a$ cells far from each other so that $a$ copies of the placements
of $A_{1},\ldots,A_{k}$ are created. The common cell $x$ of the
$k$ placements is marked by the maker in each of these copies. The
breaker cannot mark a cell in each of the $ak$ copies of the goal
animal, so the maker can win in the second turn. 
\end{proof}
The previous result is often used with $k=\delta$ as the following
example shows.
\begin{example}
\label{exa:T31}Figure~\ref{fig:SchProofSeqB} shows that the polyiamond
$T_{3,1}$ is a $(2,5)$-winner using the strategy of Proposition~\ref{pro:twostep}
with $k=\delta=3$.
\end{example}

\section{\label{sec:TheACBgame}The $(a\a c,b)$ game}

Now we introduce a variation of the regular $(a,b)$ game that we
call the $(a\a c,b)$ game. In this variation the maker marks $a$
previously unmarked cells in each turn until the very last turn. In
the last turn he is allowed to mark $c$ cells. The breaker marks
$b$ previously unmarked cells in each turn.

An animal is called a \emph{bounded winner} if the maker has a winning
strategy consisting of at most a fixed number of moves on some finite
subboard of the playing board. All the known winning polyforms are
also finite winners but see \cite[Section 5.4]{beck:book} for examples
of hypergraph games that are finite but unbounded in time or in space.

The reason for the study of the seemingly unnatural $(a\a c,b)$ games
is the following result. The proof uses some ideas found in \cite[Section 14, Section 30]{beck:book}.
We use the notation ${\bf W}={\bf N}\cup\{0\}$ for the set of whole
numbers.
\begin{thm}
\label{thm:main}Let $a=\sum_{i=1}^{s}a_{i}$ and $b=\sum_{i=1}^{s}b_{i}$
with $a_{i},b_{i}\in{\bf W}$. If a goal animal is a bounded $(a_{i}\a a,b_{i})$-winner
for all $i$, then it is also an $(a,b+s-1)$-winner.\end{thm}
\begin{proof}
Let us call the $(a_{i}\a a,b_{i})$ game the $i$-th game. The goal
animal is a bounded winner in the $i$-th game for each $i$, that
is, the maker can mark a copy of the goal animal after $l_{i}$ turns
on some sufficiently large but finite subset $B_{i}$ of the original
infinite board. Since $s$ and the $B_{i}$ are finite, we can find
a finite subset $B$ of the original infinite board that contains
a copy of every $B_{i}$. The maker can win the $i$-th game for all
$i$ on any placement of $B$ on the infinite board. Any breaker mark
outside a placement of $B$ has no effect on the outcome of the $i$-th
game played on that placement. The main idea of the maker's $(a,b+s-1)$
strategy is to mark $s$ groups of cells containing $(a_{1},\ldots,a_{s})$
cells respectively far away from each other in subboards and to play
distributed $i$-th games on the subboards that have at most $b_{i}$
breaker marks in each turn. The subboards are disjoint placements
of $B$.

The strategy for the $(a,b+s-1)$ game has stages consisting of several
turns. In each stage, the maker tries to make a one-turn progress
in one of the $i$-th games. We keep track of the progress using a
progress vector $p_{j}$ in ${\bf W}^{s}$. At the beginning of the
game, the progress vector is $p_{0}=(0,\ldots,0)$ indicating that
no progress is made in any of the subgames. 

The first stage contains $n_{1}$ turns in which the maker marks $n_{1}a$
cells. In each turn, $s$ new disjoint subboards congruent to $B$
are created. The new subboards receive $a_{1},\ldots,a_{s}$ marks
respectively, according to the first moves in the corresponding $(a_{i}\a a,b_{i})$
strategies. Let us call the subboards receiving maker moves according
to the $i$-th strategy type $i$ subboards. The total number of subboards
at the end of the stage is $n_{1}s$ since there are $n_{1}$ subboards
for each type. At the end of the first stage, we also have $n_{1}(b+s-1)$
breaker marks on the board. Let $k_{1,i}$ be the number of type $i$
alive subboards containing at most $b_{i}$ breaker marks. Then we
have
\[
n_{1}(b+s-1)\ge\sum_{i=1}^{s}(n_{1}-k_{1,i})(b_{i}+1)
\]
which implies $\sum_{i=1}^{s}k_{1,i}(b_{i}+1)\ge n_{1}$. So a large
enough $n_{1}$ guarantees that $k_{1,i_{1}}$ is large enough for
some $i_{1}$, that is, the maker can create as many alive type $i_{1}$
subboards as he needs. Now the maker disregards all the subboards
except the $k_{1,i_{1}}$ alive subboards of type $i_{1}$ that have
at most $b_{i_{1}}$ breaker marks. The progress vector becomes $p_{1}=(0,\ldots,0,1,0,\ldots,0)$
where the $1$ is at the $i_{1}$-th coordinate indicating that the
maker made progress in the $i_{1}$-th game.

Stage $j$ contains $n_{j}$ turns in which the maker marks $n_{j}a$
cells. The cells are played on $s$ different subboards. For each
$i\in\{1,\ldots,s\}$ the maker marks $a_{i}$ cells on a type $i$
alive subboard according to the strategy for the $p_{j-1}(i)+1$-st
move in the $i$-th game. If $p_{j-1}(i)=0$ then the maker creates
a new empty subboard which is a disjoint placement of $B$. Repeating
the counting argument above gives
\begin{equation}
\sum_{i=1}^{s}k_{j,i}(b_{i}+1)\ge n_{j},\label{eq:enough}
\end{equation}
that is, a large enough $n_{j}$ guarantees that for at least one
type $i_{j}$ the number $k_{j,i_{j}}$ of alive subboards of type
$i_{j}$ that have at most $b_{i_{j}}$ breaker marks in the current
stage is as large as needed. Now the maker disregards all the type
$i_{j}$ subboards except the $k_{j,i_{j}}$ alive subboards. The
progress vector $p_{j+1}$ becomes $p_{j}$ with the $i_{j}$-th coordinate
incremented, indicating that the maker made progress in the $i_{j}$-th
game.

The game continues the same way. In each stage one coordinate of the
progress vector is incremented. Eventually, say after at most $l-1$
stages, the progress vector is going to have a coordinate $i$ such
that $p_{l-1}(i)=l_{i}-1$. Then in stage $l$ the maker can mark
$a$ cells in a type $i$ alive subboard and win. So the last stage
contains only $n_{l}=1$ turn.

For this to work, the number of alive subboards of each type needs
to be large enough, so that the maker can mark cells in an alive subboard
of the type determined by the strategy. The number of required alive
subboards of each type is potentially very large but is finite since
the length $l_{i}$ of the $i$-th game is finite for all $i$. So
the maker can create sufficiently many subboards by playing each stage
long enough, that is, picking a large enough $n_{j}$ for all $j$. 

\begin{figure}
\usetikzlibrary{trees}
\begin{tikzpicture}[level distance = 10mm, sibling distance = 36mm] 
\node(root) {$\left[\begin{smallmatrix}  p(1) & \cdots & p(s) \\ q(1) & \cdots & q(s) \end{smallmatrix}\right]$}
  child{ 
    child {node{$\left[\begin{smallmatrix}  p^{(1)}(1) & \cdots & p^{(1)}(s) \\ q^{(1)}(1) & \cdots & q^{(1)}(s) \end{smallmatrix}\right]$}}
    child {node{$\left[\begin{smallmatrix}  p^{(2)}(1) & \cdots & p^{(2)}(s) \\ q^{(2)}(1) & \cdots & q^{(2)}(s) \end{smallmatrix}\right]$}}
    child {node{$\cdots$}}
    child {node{$\left[\begin{smallmatrix}  p^{(s)}(1) & \cdots & p^{(s)}(s) \\ q^{(s)}(1) & \cdots & q^{(s)}(s) \end{smallmatrix}\right]$}}
    edge from parent node[left] {$\scriptscriptstyle n$}
    }
  ;
\end{tikzpicture}

\caption{\label{fig:StagesNode}A vertex and its $s$ descendants in the stage
diagram.}
\end{figure}
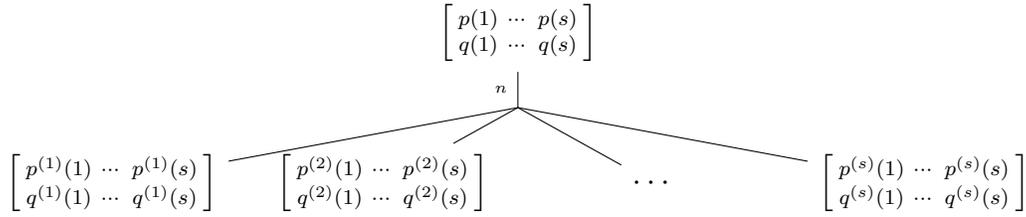

The number $n_{j}$ of required turns in each stage can be calculated
using a \emph{stage diagram.} Each vertex of the diagram is a $2\times s$
matrix that represents a possible stage during the game play. The
first row of the matrix is the progress vector $p$. The second row
is the \emph{supply vector} $q$. The $i$-th coordinate $q(i)$ of
the supply vector is the minimum number of required type $i$ alive
subboards containing at most $p(i)b_{i}$ breaker marks and $p(i)a_{i}$
maker marks according to the winning strategy for the $i$-th game. 

The game starts at the top of the diagram and progresses towards the
bottom along the edges. Each vertex has $s$ descendants. To get the
progress vector of the $i$-th descendant, we increment the $i$-th
coordinate of the progress vector of the parent vertex. The game ends
when we reach a stage with $p(i)=l_{i}$ for some $i\in\{1,\ldots,s\}$.
The labels on the edges show the number of required turns between
stages. 

A dashed edge indicates a single turn in which the maker reaches a
winning stage by putting all his $a$ marks into a single type $i$
subboard with progress $p(i)=l_{i}-1$. In the winning stages the
supply vector satisfies
\[
q(i)=\begin{cases}
1 & \text{if }p(i)=l_{i}\\
0 & \text{if }p(i)<l_{i}
\end{cases}
\]
since the maker only needs a single winning subboard of any type to
finish the game. The supply vectors of the other stages and the labels
on the edges can be calculated recursively from the bottom of the
diagram to the top. Let 
\[
\left[\begin{matrix}p^{(1)}\\
q^{(1)}
\end{matrix}\right],\ldots,\left[\begin{matrix}p^{(s)}\\
q^{(s)}
\end{matrix}\right]
\]
be the descendants of stage $v=\left[\begin{smallmatrix}p\\
q
\end{smallmatrix}\right]$ as shown in Figure~\ref{fig:StagesNode}. Then the number of required
turns after stage $v$ is 
\[
n:=1+\sum_{i=1}^{s}(q^{(i)}(i)-1)(b_{i}+1).
\]
Equation~\ref{eq:enough} guarantees that after $n$ turns the maker
reaches the $i$-th descendant with $q^{(i)}(i)$ type $i$ alive
subboards for some $i$. We write this number as a label on the edge
emanating from vertex $v$. During these $n$ turns the maker uses
up $n$ subboards of each type and the breaker marks $n(b+s-1)$ cells,
each of which mark can ruin a subboard. So to have enough supply of
alive subboards we let 
\[
q(i):=\max(\{n\}\cup\{q^{(j)}(i)+n+n(b+s-1)\mid j\in\{1,\ldots,s\}\setminus\{i\}\text{ and }q^{(j)}(i)>0\})
\]
for all $i$. We have infinitely many subboards with no progress so
we never run out of subboards. 
\end{proof}
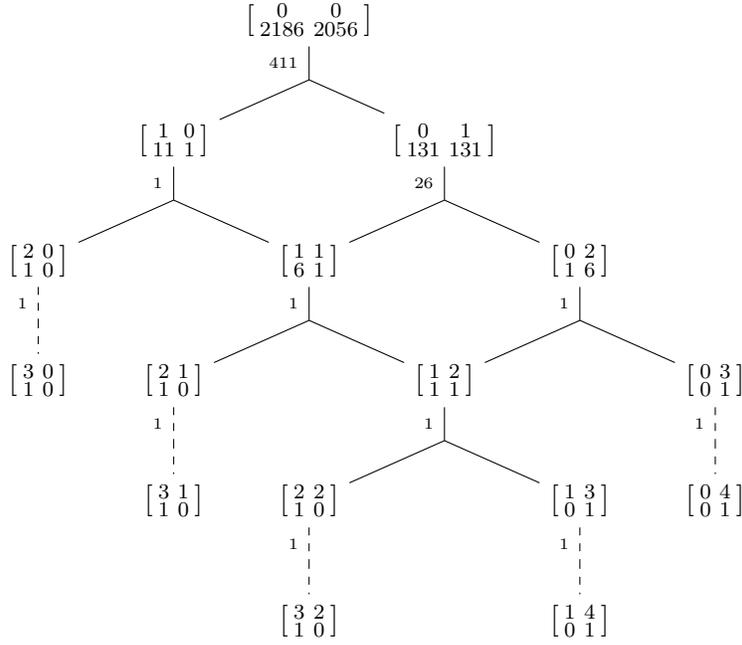
\begin{figure}
\usetikzlibrary{trees}
\begin{tikzpicture}[level distance = 8mm, sibling distance = 36mm] 
\node(root) {$\left[\begin{smallmatrix}  0 & 0 \\ 2186 & 2056 \end{smallmatrix}\right]$}
  child{ 
    child {node{$\left[\begin{smallmatrix}  1 & 0 \\ 11 & 1 \end{smallmatrix}\right]$}
      child {
        child {node{$\left[\begin{smallmatrix}  2 & 0 \\ 1 & 0 \end{smallmatrix}\right]$}
          child {
            child {node{$\left[\begin{smallmatrix}  3 & 0 \\ 1 & 0 \end{smallmatrix}\right]$}}
            edge from parent[dashed] node[left] {$\scriptscriptstyle 1$}             
          }
        }
        child {node{$\left[\begin{smallmatrix}  1 & 1 \\ 6 & 1 \end{smallmatrix}\right]$}
          child {
            child {node{$\left[\begin{smallmatrix}  2 & 1 \\ 1 & 0 \end{smallmatrix}\right]$}
              child {
                child {node{$\left[\begin{smallmatrix}  3 & 1 \\ 1 & 0 \end{smallmatrix}\right]$}}
                edge from parent[dashed] node[left] {$\scriptscriptstyle 1$}             
              }
            }
            child {node{$\left[\begin{smallmatrix}  1 & 2 \\ 1 & 1 \end{smallmatrix}\right]$}
              child {
                child{node{$\left[\begin{smallmatrix}  2 & 2 \\ 1 & 0 \end{smallmatrix}\right]$}
                  child {
                    child {node{$\left[\begin{smallmatrix}  3 & 2 \\ 1 & 0 \end{smallmatrix}\right]$}}
                    edge from parent[dashed] node[left] {$\scriptscriptstyle 1$}             
                  }
                }
                child{node{$\left[\begin{smallmatrix}  1 & 3 \\ 0 & 1 \end{smallmatrix}\right]$}
                  child {
                    child {node{$\left[\begin{smallmatrix}  1 & 4 \\ 0 & 1 \end{smallmatrix}\right]$}}
                  edge from parent[dashed] node[left] {$\scriptscriptstyle 1$}             
                  }
                }
                edge from parent node[left] {$\scriptscriptstyle 1$}
              }
            }
            edge from parent node[left] {$\scriptscriptstyle 1$}
          }
        }
        edge from parent node[left] {$\scriptscriptstyle 1$}
      }
    }
    child {node{$\left[\begin{smallmatrix}  0 & 1 \\ 131 & 131 \end{smallmatrix}\right]$}
      child {
        child[missing] {}
        child {node{$\left[\begin{smallmatrix}  0 & 2 \\ 1 & 6 \end{smallmatrix}\right]$}
          child {
            child[missing] {}
            child{node{$\left[\begin{smallmatrix}  0 & 3 \\ 0 & 1 \end{smallmatrix}\right]$}
              child {
                child {node{$\left[\begin{smallmatrix}  0 & 4 \\ 0 & 1 \end{smallmatrix}\right]$}}
                edge from parent[dashed] node[left] {$\scriptscriptstyle 1$}             
              }
            } 
            edge from parent node[left] {$\scriptscriptstyle 1$}
          }
        }
        edge from parent node[left] {$\scriptscriptstyle 26$}
      }
    }
    edge from parent node[left] {$\scriptscriptstyle 411$}
  }
  ;
  \draw (root-1-2-1)-- (root-1-1-1-2);
  \draw (root-1-2-1-2-1)-- (root-1-1-1-2-1-2);
\end{tikzpicture}

\caption{\label{fig:Stage-diagram}A stage diagram for $b_{1}=1$,\textbf{
$b_{2}=2$, $l_{1}=3$} and $l_{2}=4$.}
\end{figure}

\begin{example}
Figure~\ref{fig:Stage-diagram} shows a stage diagram for a game
with $s=2$, $b=b_{1}+b_{2}=1+2$, $l_{1}=3$ and $l_{2}=4$. We show
the details of the calculation at the stage with $p=(1,0)$ which
is the first vertex on the second row. We have 
\[
\left[\begin{matrix}p^{(1)}\\
q^{(1)}
\end{matrix}\right]=\left[\begin{matrix}2 & 0\\
1 & 0
\end{matrix}\right],\qquad\left[\begin{matrix}p^{(2)}\\
q^{(2)}
\end{matrix}\right]=\left[\begin{matrix}1 & 1\\
6 & 1
\end{matrix}\right],
\]
\[
n=1+(q^{(1)}(1)-1)(b_{1}+1)+(q^{(2)}(2)-1)(b_{2}+1)=1.
\]
Thus $q(1)=\max\{n,q^{(2)}(1)+n(b+s)\}=\max\{1,11\}=11$ and $q(2)=\max\{n\}=1$.
Note that the game finishes after at most 440 turns.
\end{example}
The number of required turns calculated using the stage diagram is
potentially very large. Our calculation is a crude overestimate that
could be improved with a more complicated argument but our goal was
only to prove that the strategy works. Also note that the breaker
can use her marks to ruin many subboards by ignoring some of the subboards
completely. In this case, the maker has subboards with very few defensive
marks on them and so he probably can win faster using a more refined
strategy.
\begin{cor}
\label{cor:reduce}If an animal is a $(1\a a,\lfloor\frac{b}{a}\rfloor)$-winner,
then it is also an $(a,b)$-winner.\end{cor}
\begin{proof}
 The result is a special case of Theorem~\ref{thm:main} with $a_{i}=1$
and $b_{i}=\lfloor\frac{b}{a}\rfloor$ for all $i\in\{1,\ldots,a\}$.
\end{proof}
Note that finding a proof sequence for the $(1\a a,\lfloor\frac{b}{a}\rfloor)$
game is often much simpler than for the $(a,b)$ game. The following
is an easy consequence.
\begin{cor}
\label{cor:bLEa}If $b<a$ then any animal is an $(a,b)$-winner.
\end{cor}

\section{\label{sec:The-priority-strategy}The priority strategy}

One of the difficulties of the theory of achievement games is the
lack of strategies for the breaker other than the pairing strategy
based on pavings. In this section we describe a new strategy for the
breaker.
\begin{defn}
An \emph{$(a,b)$ priority strategy} is a strategy for the breaker.
Let $(x_{1},\ldots,x_{a})$ be an ordering of the current marks of
the maker. The ordering can depend for example on the location of
the marked cell or on the relative positions of the marks. By default
we order them using the lexicographic order of the coordinates. The
priority strategy assigns a \emph{response set} $R_{x_{i}}$ of \emph{response
cells} for each $x_{i}$. The priority of the response cells in $R_{x_{i}}$
is determined by a \emph{priority number}. A smaller number means
higher priority. In the simplest case, the priority numbers of the
response cells are the same for each mark of the breaker. In more
complicated cases, the priority numbers of the response cells may
depend on the location or on the ordering of the maker marks. The
priority numbers in a response set can change after each breaker mark.
The breaker tries to mark one of the highest priority unmarked cells
in each of the response sets following the order $R_{x_{1}},\ldots,R_{x_{a}},R_{x_{1}},\ldots$
until she runs out of marks. If there are no unmarked cells in a response
set, then the breaker moves to the next response set. If the breaker
is not able to use all her $b$ marks, then she can use the remaining
marks on random cells. 
\end{defn}
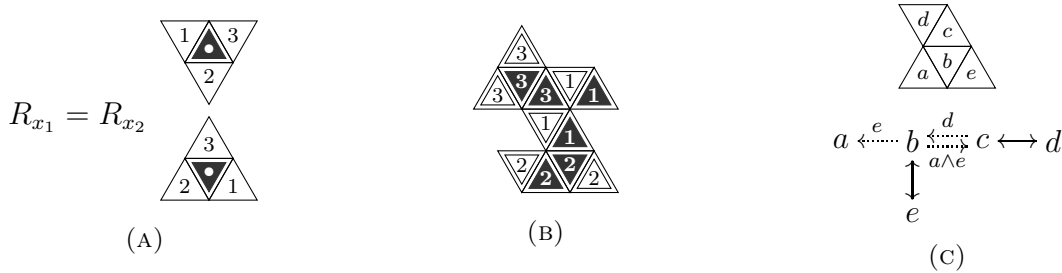
\begin{figure}
\renewcommand\trrad{.37}

\hfill{}\subfloat[\label{fig:priorityStrA}]{$R_{x_{1}}=R_{x_{2}}$%
\begin{tabular}{c}
\begin{tikzpicture}
\trUMC{0}{1}{\bullet}
\trD{-1}{1}{1}{}{}{}
\trD{0}{1}{3}{}{}{}
\trD{0}{0}{2}{}{}{}
\end{tikzpicture}\tabularnewline
\begin{tikzpicture}
\trDMC{0}{0}{\bullet}
\trU{0}{0}{2}{}{}{}
\trU{1}{0}{1}{}{}{}
\trU{0}{1}{3}{}{}{}
\end{tikzpicture}\tabularnewline
\end{tabular}

}\hfill{}\subfloat[\label{fig:priorityStrB}]{%
\begin{tabular}{c}
\begin{tikzpicture}
\trUBC{-1}{1}{3}
\trUBC{-1}{2}{3}
\trDMC{-1}{1}{\text{\bf 3}}
\trUMC{0}{1}{\text{\bf 3}}
\trDBC{0}{0}{1}
\trDBC{0}{1}{1}
\trUMC{1}{0}{\text{\bf 1}}
\trUMC{1}{1}{\text{\bf 1}}
\trDMC{1}{-1}{\text{\bf 2}}
\trUMC{1}{-1}{\text{\bf 2}}
\trDBC{0}{-1}{2}
\trUBC{2}{-1}{2}
\end{tikzpicture}\tabularnewline
\end{tabular}}\hfill{}\subfloat[\label{fig:priorityStrC}]{%
\begin{tabular}{c}
\begin{tikzpicture}
\trU{-1}{1}{a}{}{}{}
\trU{-1}{2}{c}{}{}{}
\trD{-1}{1}{b}{}{}{}
\trD{-2}{2}{d}{}{}{}
\trU{0}{1}{e}{}{}{}
\end{tikzpicture}\tabularnewline
\xymatrix@=5mm{
a & b \ar@{<->}[d] \ar@{.>}_e[l] \ar@<-2pt>@{.>}_{a \land e}[r] & c \ar@{<->}[r] \ar@<-2pt>@{.>}_d[l] & d \\
  & e
}\tabularnewline
\end{tabular}}\hfill{}

\caption{(a) A $(2,2)$ priority strategy. The bullets indicate the marks of
the maker in the current turn. (b) A possible game play where the
breaker follows the priority strategy. The marks of the maker are
black while the marks of the breaker are white. (c) A dependency digraph
of a placement of a goal animal. }
\end{figure}

Note that every paving strategy is a priority strategy where the cells
with priority one are guaranteed to be available.
\begin{example}
Figure~\ref{fig:priorityStrA} shows a graphical description of a
$(2,2)$ priority strategy. For each mark of the maker, the breaker
has three possible response cells. The priorities are the same for
both maker marks but they depend on the type of the cell. The breaker
first tries to mark the cells containing a $1$. If any of those cells
are not available because they are already marked, then she tries
to mark the cells containing a $2$ and so on. If none of these cells
are available, then she marks random cells. Figure~\ref{fig:priorityStrB}
shows a possible position as a result of the priority strategy. The
breaker marks a priority 2 and a priority 3 cell in the last turn
since no priority 1 cell is available.
\end{example}
Note that the graphical representations of priority strategies are
generally not invariant under reflections or rotations.
\begin{example}
Figure~\ref{fig:24P41A} shows a graphical description of a $(2,4)$
priority strategy where the priority numbers in a response set are
not constant. On her first mark in the response set, the breaker tries
to mark cell $\alpha_{1}$. She marks $\alpha_{2}$ if $\alpha_{1}$
is already marked. On her second mark in the same response set, the
breaker tries to mark cell $\beta_{1}$ and falls back on $\beta_{2}$
if $\beta_{1}$ is already marked. \end{example}
\begin{defn}
Given a priority strategy and a placement of the goal animal we have
a \emph{dependency relation} on the set of cells of the goal animal.
We write $a\le b$ if the maker cannot achieve the placement of the
goal animal if he tries to mark cell $a$ in turn that comes later
than the turn in which he marks cell $b$. We write $a\sim b$ if
$a$ and $b$ have to be marked by the maker in the same turn, that
is, $a\le b$ and $b\le a$. 
\end{defn}
It is clear that the dependency relation is transitive.
\begin{defn}
The dependency relation can be captured using a \emph{dependency digraph}.
The vertex set of the digraph is the set of cells of the goal animal.
We use three types of arrows:
\begin{itemize}
\item Arrow $\xymatrix@1@=6mm{a \ar[r] & b}$ is called an \emph{unconditional
arrow.} It indicates that cell $b$ cannot be marked after cell $a$.
We add this arrow when response cell $b$ in $R_{a}$ has maximal
priority, that is, a priority in the set $\{1,\alpha_{1},\beta_{1},\ldots\}$.
An unmarked cell $b$ in this situation is going to be marked by the
breaker right after the maker marks cell $a$. 
\item Arrow $\xymatrix@1@=11mm{a \ar@{.>}^{l_1,\ldots,l_k}[r] & b}$ is
called a \emph{conditional arrow}. It indicates that cell $b$ cannot
be marked after cell $a$ if condition $l_{i}$ is satisfied for some
$i$. Each condition is of the form $l_{i}=x_{1}\land\cdots\land x_{r}$
where $x_{j}$ is a logical variable corresponding to a cell of the
goal animal for all $j$. The value of $x_{j}$ is true if cell $x_{j}$
is marked by either the maker or the breaker in a turn no later than
the turn in which $a$ is marked. We add this arrow when response
cell $b$ in $R_{a}$ does not have maximal priority but sufficiently
many marked response cells $x_{1},\ldots,x_{r}$ in $R_{a}$ with
higher priority makes the priority of $b$ large enough to guarantee
the marking of $b$. In this situation, cell $b$ left unmarked by
the maker is going to be marked by the breaker right after the maker
marks cell $a$ in spite of the lower priority.
\item Arrow $\xymatrix@1@=11mm{a \ar@{.>>}^{l_1,\ldots,l_k}[r] & b}$ is
called a \emph{secondary arrow}. This arrow is similar to the conditional
arrow but condition $x_{1}\land\cdots\land x_{r}$ is interpreted
differently. We add this arrow when response cell $b$ in $R_{a}$
does not have maximal priority but sufficiently many response cells
in $R_{a}$ with higher priority are already marked because these
response cells are also maximal priority response cells for one of
the cells $x_{1},\ldots,x_{r}$ already marked by the maker. So the
value of $x_{j}$ is true if cell $x_{j}$ is marked in a turn earlier,
then the turn in which $a$ is marked. The value of $x_{j}$ can be
true even if cells $x_{j}$ and $a$ are marked at the same turn.
For this to happen, $x_{j}$ and $a$ have to be ordered in the set
of current marks by the priority strategy to make the common response
cell $z$ of $x_{j}$ and $a$ an earlier response cell for $x_{j}$
than a response cell for $a$. 
\end{itemize}
We might omit secondary arrows and even conditional arrows in our
dependency digraphs if they are not needed.\end{defn}
\begin{example}
Figure~\ref{fig:priorityStrC} shows a goal animal and its dependency
digraph based on the priority strategy presented in the same figure.
The arrow $\xymatrix@1@=11mm{b \ar@{.>}^{a\land e}[r] & c}$ has one
label. This arrow indicates that if the maker marks cell $b$ but
leaves cell $c$ unmarked, then the breaker is going to mark cell
$c$ assuming cells $a$ and $e$ are already marked. 

If the maker wants to mark all the cells in $\{a,b,c,d,e\}$, then
he needs to make sure that in each turn he marks all the cells that
are dependent on other marks of his own. For example he can mark cell
$a$ without marking any other cells of the goal animal. On the other
hand if he marks cell $b$, then he has to mark all the other cells
in the same turn. To see this first note that there is a solid arrow
from $b$ to $e$ so cell $e$ needs to be marked. This implies that
cell $a$ needs to be marked since the label of the dotted arrow is
satisfied. Now the label on the arrow from $b$ to $c$ is satisfied
so cell $c$ must be marked as well. Finally cell $d$ needs to be
marked since there is a solid arrow from $c$ to $d$. Thus we have
$a\le b\sim c\sim d\sim e$.

It is easy to see that there are only two options for marking all
the cells in $\{a,b,c,d,e\}$. The first option is to mark cell $a$
in a turn and then mark the rest of the cells in a later turn. The
other option is to mark all the cells in a single turn. Both of these
options are impossible since the maker can only mark two cells in
a turn.
\end{example}

\begin{example}
\label{exa:secondary}Figure~\ref{fig:24P43B} shows a dependency
digraph containing secondary arrows. The secondary arrows exist because
the response cell for $a$ with priority $\beta_{1}$ is the same
as the response cell for $d$ with priority $\alpha_{1}$. This common
response cell $x$ is located above $a$ and on the left of $d$.
We clearly have $a,d\le b\sim c$. If the maker marks $a$ before
$d$, then $x$ is marked as a response cell for $a$ and so the vertical
secondary arrow from $d$ to $b$ is activated implying $a\le b\sim c\sim d$.
If the maker marks $d$ before $a$, then $x$ is marked as a response
cell for $d$ and so the horizontal secondary arrow from $a$ to $b$
is activated implying $d\le a\sim b\sim c$. 

The situation is a bit more delicate if the maker marks $a$ and $d$
in the same turn. In the lexicographic order $a$ is considered smaller
than $d$ so the breaker first marks a response cell for $a$. This
is the response cell on the left of $a$ with priority $\alpha_{1}$.
Next the breaker marks a response cell for $d$. This response cell
is $x$ with priority $\alpha_{1}$. Hence the horizontal secondary
arrow from $a$ to $b$ is activated and we have $d\le a\sim b\sim c$.

In all three cases at least three cells have to be marked in a single
turn.
\end{example}

\section{\label{sec:The-threshold-sequence}The threshold sequence}

Our main goal is to determine all the $(a,b)$ pairs for which an
animal is an $(a,b)$-winner. We use the following object to capture
this information.
\begin{defn}
The \emph{threshold sequence} $\tau(A)=(b_{1},b_{2},\ldots)$ for
a given animal $A$ is a sequence of numbers in ${\bf W}\cup\{\infty\}$
such that $b_{n}$ is the greatest value for which $A$ is an $(n,b_{n})$\emph{-winner}.
\end{defn}
The following is an easy consequence of the fact that marking more
cells could never hurt the players. 
\begin{lem}
\label{lem:morea}Let $A$ be an animal and $n,k\in{\bf W}$. If $A$
is an $(a,b)$-winner, then it is also an $(a+n,b-k)$-winner. If
$A$ is an $(a,b)$-loser, then it is also an $(a-n,b+k)$-loser.\end{lem}
\begin{prop}
For every threshold sequence $(b_{1},b_{2},\ldots)$ there is an index
$k\in{\bf W}$ such that $b_{n}$ is finite for all $n<k$ and $b_{n}=\infty$
for all $n\ge k$.\end{prop}
\begin{proof}
If the animal $A$ has $l$ cells, then $A$ is clearly an $(l,b)$-winner
for all $b\in{\bf W}$. If $A$ is an $(k,\infty)$-winner, then it
is also an $(n,b)$-winner for all $n\ge k$ and $b\in{\bf W}$.
\end{proof}
We are going to write the threshold sequence $(b_{1},\ldots,b_{k-1},\infty,\ldots)$
where $b_{k-1}<\infty$ simply as $(b_{1},\ldots,b_{k-1},\infty)$.
\begin{prop}
In a threshold sequence $(b,\ldots,b_{k-1},\infty)$ we have $b_{i}<b_{i+1}$
for all $i\in\{1,\ldots,k-1\}$.\end{prop}
\begin{proof}
It is clear that every animal is a $(1,0)$-winner. Since the animal
is also an $(i,b_{i})$-winner, it must be a $(i+1,b_{i}+1)$-winner
by Theorem~\ref{thm:main}.\end{proof}
\begin{prop}
If $A$ is a subset of animal $B$, then $\tau(A)(i)\ge\tau(B)(i)$
for all $i\in{\bf N}$.\end{prop}
\begin{proof}
Any successful maker strategy for $B$ is also successful for $A$
. Hence if $B$ is an $(a,b)$-winner, then so is $A$.
\end{proof}
The following is an easy consequence of Corollary~\ref{cor:bLEa}.
\begin{prop}
For all animal $A$ we have $\tau(A)(i)\ge i-1$.
\end{prop}

\section{\label{sec:Polyiamonds}Polyiamonds}

\begin{figure}
\begin{tabular}{c|ccccccccc}
 & \begin{tikzpicture}
\trU{0}{0}{}{}{}{}
\end{tikzpicture} &  & \begin{tikzpicture}
\trU{0}{0}{}{}{}{}
\trD{0}{0}{}{}{}{}
\end{tikzpicture} &  & \begin{tikzpicture}
\trU{0}{0}{}{}{}{}
\trD{0}{0}{}{}{}{}
\trU{1}{0}{}{}{}{}
\end{tikzpicture} &  & \begin{tikzpicture}
\trU{0}{0}{}{}{}{}
\trD{0}{0}{}{}{}{}
\trU{1}{0}{}{}{}{}
\trD{1}{0}{}{}{}{}
\end{tikzpicture} & \begin{tikzpicture}
\trU{0}{0}{}{}{}{}
\trD{0}{0}{}{}{}{}
\trD{-1}{1}{}{}{}{}
\trU{0}{1}{}{}{}{}
\end{tikzpicture} & \begin{tikzpicture}
\trU{0}{0}{}{}{}{}
\trD{0}{0}{}{}{}{}
\trU{1}{0}{}{}{}{}
\trU{0}{1}{}{}{}{}
\end{tikzpicture}\tabularnewline
$\vphantom{\int_{\int}}A$ & $T_{1,1}$ &  & $T_{2,1}$ &  & $T_{3,1}$ &  & $T_{4.1}$ & $T_{4,2}$ & $T_{4,3}$\tabularnewline
\hline 
$\vphantom{\int^{\int}}\tau(A)$ & $(\infty)$ &  & $(2,\infty)$ &  & $(1,5,\infty)$ &  & $(0,3,8,\infty)$ & $(0,3,8,\infty)$ & $(0,3,8,\infty)$\tabularnewline
\end{tabular}\caption{\label{fig:polyiamonds}Polyiamonds and their threshold sequences
up to size four.}
\end{figure}

Unbiased polyiamond games are studied in \cite{bode.harborth:triangular,bode.harborth:triangle,sieben:sitePerimeter}.
In this section we find the threshold sequences of polyiamonds up
to size four. The results are summarized in Figure~\ref{fig:polyiamonds}.
The proof of the following result is left to the reader.
\begin{prop}
The threshold sequences of $T_{1,1}$ and $T_{2,1}$ are $\tau(T_{1,1})=(\infty)$
and $\tau(T_{2,1})=(2,\infty)$.
\end{prop}
We now consider the size three polyiamond.

\begin{figure}
\renewcommand\trrad{.45}

\hfill{}\subfloat[\label{fig:proofSeq11T32}]{%
\begin{tabular}{ccc}
$s_{0}$ & $s_{1}$ & $s_{2}$\tabularnewline
\begin{tikzpicture}
\trUM{0}{0}
\trDM{0}{0}
\trUM{1}{0}
\end{tikzpicture} & \begin{tikzpicture}
\trD{0}{0}{A}{}{}{}
\trU{2}{0}{B}{}{}{}
\trUM{1}{0}
\trDM{1}{0}
\end{tikzpicture} & \begin{tikzpicture} 
\trUM{1}{0}
\trD{1}{0}{Ab}{}{}{}
\trD{0}{0}{Ca}{}{}{}
\trU{0}{1}{c}{}{}{}
\trU{2}{0}{a}{}{}{}
\trD{1}{-1}{Bc}{}{}{}
\trU{1}{-1}{b}{}{}{}
\end{tikzpicture}\tabularnewline
\end{tabular}\lower-0.9cm\hbox{
$\xymatrix@=5mm{
s_2 \ar[d]_{a,b,c} \\
s_1 \ar[d]_{a,b}  \\
s_0 \\
}$
}

}\hfill{}\renewcommand\trrad{.35}\subfloat[\label{fig:23PrSeT41}]{%
\begin{tabular}{ccc}
$s_{0}$ &  & $s_{1}$\tabularnewline
\begin{tikzpicture}
\trUM{0}{0}
\trDM{0}{0}
\trUM{1}{0}
\trDM{1}{0}
\end{tikzpicture} &  & \begin{tikzpicture}
\trD{0}{0}{A}{}{}{}
\trU{0}{1}{A}{}{}{}
\trD{1}{0}{B}{}{}{}
\trU{1}{1}{B}{}{}{}
\trD{-1}{2}{C}{}{}{}
\trU{-1}{3}{C}{}{}{}
\trD{0}{2}{D}{}{}{}
\trU{0}{3}{D}{}{}{}
\trUM{0}{2}
\trDM{0}{1}
\end{tikzpicture}\tabularnewline
\end{tabular}\lower-0.3cm\hbox{
$\xymatrix@=5mm{
s_1 \ar[d]_{a,b,c,d}  \\
s_0 \\
}$
}}\hfill{}

\caption{(a) A $(1,1)$ proof sequence for $T_{3,1}$. (b) A $(2,3)$ proof
sequence for $T_{4,1}$. }
\end{figure}
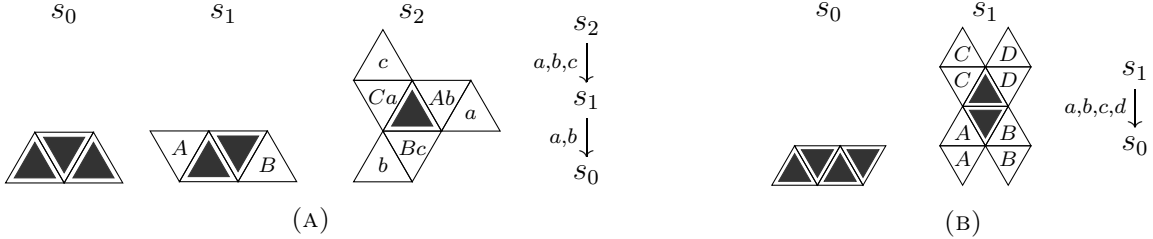

\begin{figure}
\begin{tabular}{ccc}
\renewcommand\trrad{.3}
\begin{tikzpicture} 
\trTileUD{0}{1}{0}{0}
\trTileUD{0}{2}{0}{1}
\trTileUD{1}{1}{1}{0}
\trTileUD{1}{2}{1}{1}
\end{tikzpicture} & \renewcommand\trrad{.3}
\begin{tikzpicture} 
\trTileUD{0}{1}{0}{0}
\trTileUD{-1}{1}{-1}{1}
\trTileUD{1}{1}{0}{1}
\trTileUD{0}{2}{0}{2}
\trTileUD{-1}{2}{-2}{2}
\trTileUD{-1}{3}{-1}{2}
\trTileUD{1}{2}{1}{1}
\trTileUD{1}{0}{1}{0}
\trTileUD{2}{1}{2}{1}
\trTileUD{3}{0}{2}{0}
\trTileUD{2}{0}{2}{-1}

\trTileUD{3}{1}{3}{0}
\trTileUD{2}{1}{2}{1}
\trTileUD{4}{1}{3}{1}
\trTileUD{3}{2}{3}{2}
\trTileUD{2}{2}{1}{2}
\trTileUD{2}{3}{2}{2}
\trTileUD{4}{2}{4}{1}
\trTileUD{4}{0}{4}{0}
\trTileUD{5}{1}{5}{1}
\trTileUD{6}{0}{5}{0}
\trTileUD{5}{0}{5}{-1}
\end{tikzpicture} & \renewcommand\trrad{.3}
\begin{tikzpicture} 
\trTileUD{1}{0}{1}{0}
\trTileUD{1}{1}{1}{0}
\trTileUD{1}{1}{0}{1}
\trTileUD{0}{1}{0}{1}
\trTileUD{1}{0}{0}{0}
\trTileUD{0}{1}{0}{0}

\trTileUD{2}{1}{2}{1}
\trTileUD{2}{2}{2}{1}
\trTileUD{2}{2}{1}{2}
\trTileUD{1}{2}{1}{2}
\trTileUD{2}{1}{1}{1}
\trTileUD{1}{2}{1}{1}

\trTileUD{3}{-1}{3}{-1}
\trTileUD{3}{0}{3}{-1}
\trTileUD{3}{0}{2}{0}
\trTileUD{2}{0}{2}{0}
\trTileUD{3}{-1}{2}{-1}
\trTileUD{2}{0}{2}{-1}

\trTileUD{4}{0}{4}{0}
\trTileUD{4}{1}{4}{0}
\trTileUD{4}{1}{3}{1}
\trTileUD{3}{1}{3}{1}
\trTileUD{4}{0}{3}{0}
\trTileUD{3}{1}{3}{0}
\end{tikzpicture}\tabularnewline
$\mathsf{T}_{1,1}$ & $\mathsf{T}_{1,2}$ & $\mathsf{T}_{2,1}$\tabularnewline
\end{tabular}

\caption{\label{fig:TPavings}Triangular pavings.}
\end{figure}
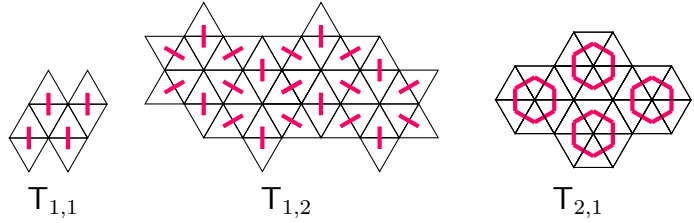

\begin{prop}
The threshold sequence of $T_{3,1}$ is $\tau(T_{3,1})=(1,5,\infty)$.\end{prop}
\begin{proof}
The proof sequence of Figure~\ref{fig:proofSeq11T32} shows that
$T_{3,1}$ is a $(1,1)$-winner. The breaker strategy based on the
2-paving $\mathsf{T}_{2,1}$ shown in Figure~\ref{fig:TPavings}
makes $T_{3,1}$ a $(1,2)$-loser. 

We saw in Example~\ref{exa:T31} that $T_{3,1}$ is a $(2,5)$-winner.
Proposition~\ref{pro:surround} implies that $T_{3,1}$ is a $(2,6)$-loser.
\end{proof}
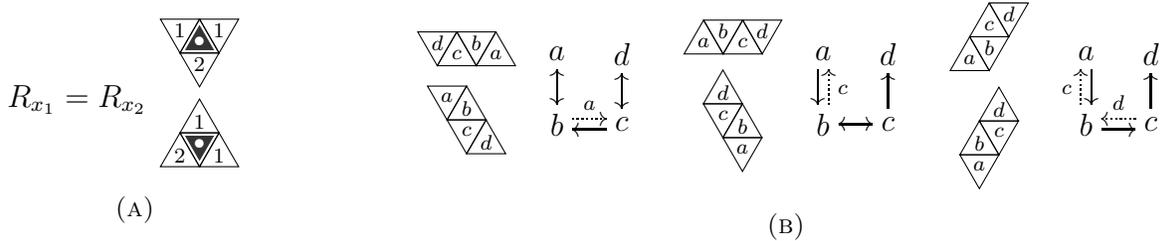
\begin{figure}
\renewcommand\trrad{.3}

\hfill{}\subfloat[\label{fig:24T41A}]{$R_{x_{1}}=R_{x_{2}}$%
\begin{tabular}{c}
\begin{tikzpicture}
\trUMC{0}{1}{\bullet}
\trD{-1}{1}{1}{}{}{}
\trD{0}{1}{1}{}{}{}
\trD{0}{0}{2}{}{}{}
\end{tikzpicture}\tabularnewline
\begin{tikzpicture}
\trDMC{0}{0}{\bullet}
\trU{0}{0}{2}{}{}{}
\trU{1}{0}{1}{}{}{}
\trU{0}{1}{1}{}{}{}
\end{tikzpicture}\tabularnewline
\end{tabular}

}\hfill{}\subfloat[\label{fig:24T41B}]{

\begin{tabular}{c}
\begin{tikzpicture}
\trD{-1}{1}{d}{}{}{}
\trU{0}{1}{c}{}{}{}
\trD{0}{1}{b}{}{}{}
\trU{1}{1}{a}{}{}{}
\end{tikzpicture}\tabularnewline
\begin{tikzpicture}
\trD{-1}{0}{a}{}{}{}
\trU{0}{0}{b}{}{}{}
\trD{0}{-1}{c}{}{}{}
\trU{1}{-1}{d}{}{}{}
\end{tikzpicture}\tabularnewline
\end{tabular}\raise0.7cm\hbox{
\xymatrix@=4mm{
a \ar@{<->}[d] & d \\
b \ar@<2pt>@{.>}^a[r] & c \ar@{<->}[u] \ar@<2pt>[l]
}
}$\quad$%
\begin{tabular}{c}
\begin{tikzpicture}
\trU{0}{1}{a}{}{}{}
\trD{0}{1}{b}{}{}{}
\trU{1}{1}{c}{}{}{}
\trD{1}{1}{d}{}{}{}
\end{tikzpicture}\tabularnewline
\begin{tikzpicture}
\trD{0}{0}{a}{}{}{}
\trU{0}{1}{b}{}{}{}
\trD{-1}{1}{c}{}{}{}
\trU{-1}{2}{d}{}{}{}
\end{tikzpicture}\tabularnewline
\end{tabular}\raise0.7cm\hbox{
\xymatrix@=4mm{
a \ar@<-2pt>[d] & d \\
b \ar@<-2pt>@{.>}_c[u] \ar@{<->}[r] & c \ar[u] 
}
}$\quad$%
\begin{tabular}{c}
\begin{tikzpicture}
\trD{0}{1}{d}{}{}{}
\trU{0}{1}{c}{}{}{}
\trD{0}{0}{b}{}{}{}
\trU{0}{0}{a}{}{}{}
\end{tikzpicture}\tabularnewline
\begin{tikzpicture}
\trD{0}{0}{a}{}{}{}
\trU{0}{1}{b}{}{}{}
\trD{0}{1}{c}{}{}{}
\trU{0}{2}{d}{}{}{}
\end{tikzpicture}\tabularnewline
\end{tabular}\raise0.7cm\hbox{
\xymatrix@=4mm{
a \ar@<2pt>[d] & d \\
b \ar@<2pt>@{.>}^c[u] \ar@<-2pt>[r] & c \ar@<-2pt>@{.>}_d[l] \ar[u] 
}
}

}\hfill{}

\caption{\label{fig:24T41}(a) A $(2,4)$ priority strategy for the breaker
for $T_{4,1}$. (b) Dependency digraphs of the cells in the orientations
of $T_{4,1}$. }
\end{figure}

\begin{prop}
The threshold sequence of $T_{4,1}$ is $\tau(T_{4,1})=(0,3,8,\infty)$.\end{prop}
\begin{proof}
The breaker strategy based on paving $\mathsf{T}_{1,2}$ shown in
Figure~\ref{fig:TPavings} makes $T_{4,1}$ a $(1,1)$-loser. It
is easy to see that $T_{4,1}$ is a $(3,8)$-winner by Proposition~\ref{pro:twostep}
and a $(3,9)$-loser by Proposition~\ref{pro:surround}. The proof
sequence of Figure~\ref{fig:23PrSeT41} shows that $T_{4,1}$ is
a $(2,3)$-winner. 

It remains to show that $T_{4,1}$ is a $(2,4)$-loser. We are going
to show that the breaker wins following the priority strategy determined
by Figure~\ref{fig:24T41A}. Figure~\ref{fig:24T41B} shows all
possible orientations of the goal animal. For each these orientations
the dependency digraph of the cells show that cells $a$, $b$ and
$c$ must be marked in the same turn by the maker to achieve the goal
animal. The maker is not able to do so since he only has two marks
in a turn. 
\end{proof}
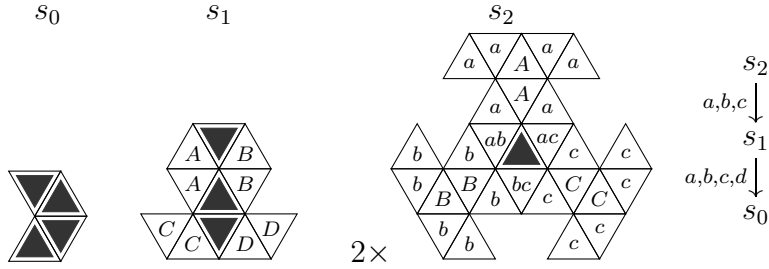
\begin{figure}
\renewcommand\trrad{.4}

\begin{tabular}{ccccc}
$s_{0}$ &  & $s_{1}$ &  & $s_{2}$\tabularnewline
\begin{tikzpicture}
\trUM{0}{0}{}{}{}{}
\trDM{0}{0}{}{}{}{}
\trDM{-1}{1}{}{}{}{}
\trUM{0}{1}{}{}{}{}
\end{tikzpicture} &  & \begin{tikzpicture}
\trD{0}{0}{A}{}{}{}
\trD{0}{-1}{C}{}{}{}
\trU{1}{-1}{C}{}{}{}
\trD{1}{0}{B}{}{}{}
\trU{2}{-1}{D}{}{}{}
\trD{2}{-1}{D}{}{}{}
\trDM{0}{1}
\trUM{1}{0}
\trDM{1}{-1}
\trU{0}{1}{A}{}{}{}
\trU{1}{1}{B}{}{}{}
\end{tikzpicture} &  & $2\times$\begin{tikzpicture} 
\trUM{1}{1}
\trU{0}{3}{A}{}{}{}
\trD{0}{3}{a}{}{}{}
\trU{1}{3}{a}{}{}{}
\trD{-1}{3}{a}{}{}{}
\trU{-1}{3}{a}{}{}{}
\trD{0}{2}{A}{}{}{}
\trU{0}{2}{a}{}{}{}
\trU{1}{2}{a}{}{}{}
\trD{0}{1}{ab}{}{}{}
\trD{1}{1}{ac}{}{}{}

\trU{0}{0}{B}{}{}{}
\trD{0}{0}{B}{}{}{}
\trU{-1}{1}{b}{}{}{}
\trU{1}{-1}{b}{}{}{}
\trU{0}{1}{b}{}{}{}
\trU{1}{0}{b}{}{}{}
\trD{-1}{0}{b}{}{}{}
\trD{0}{-1}{b}{}{}{}

\trD{2}{0}{C}{}{}{}
\trU{3}{0}{C}{}{}{}
\trD{3}{0}{c}{}{}{}
\trU{3}{1}{c}{}{}{}
\trU{3}{-1}{c}{}{}{}
\trD{3}{-1}{c}{}{}{}
\trU{2}{0}{c}{}{}{}
\trU{2}{1}{c}{}{}{}
\trD{1}{0}{bc}{}{}{}

\end{tikzpicture}\tabularnewline
\end{tabular}\lower-0.9cm\hbox{
$\xymatrix@=5mm{
s_2 \ar[d]_{a,b,c} \\
s_1 \ar[d]_{a,b,c,d}  \\
s_0 \\
}$
}

\caption{\label{fig:23PrSeT42}A $(2,3)$ proof sequence for $T_{4,2}$. }
\end{figure}

\begin{figure}
\renewcommand\trrad{.3}

\hfill{}\subfloat[\label{fig:24T42A}]{$R_{x_{1}}=R_{x_{2}}$%
\begin{tabular}{c}
\begin{tikzpicture}
\trUMC{0}{1}{\bullet}
\trD{-1}{1}{2}{}{}{}
\trD{0}{1}{1}{}{}{}
\trD{0}{0}{1}{}{}{}
\end{tikzpicture}\tabularnewline
\begin{tikzpicture}
\trDMC{0}{0}{\bullet}
\trU{0}{0}{1}{}{}{}
\trU{1}{0}{2}{}{}{}
\trU{0}{1}{1}{}{}{}
\end{tikzpicture}\tabularnewline
\end{tabular}

}\hfill{}\subfloat[\label{fig:24T42B}]{%
\begin{tabular}{cccc}
\begin{tikzpicture}
\trU{0}{0}{d}{}{}{}
\trD{0}{0}{c}{}{}{}
\trD{-1}{1}{a}{}{}{}
\trU{0}{1}{b}{}{}{}
\end{tikzpicture}~\begin{tikzpicture}
\trU{0}{0}{a}{}{}{}
\trD{-1}{0}{b}{}{}{}
\trD{-1}{1}{d}{}{}{}
\trU{-1}{1}{c}{}{}{}
\end{tikzpicture}~\begin{tikzpicture}
\trU{0}{0}{b}{}{}{}
\trD{0}{0}{c}{}{}{}
\trD{-1}{0}{a}{}{}{}
\trU{0}{1}{d}{}{}{}
\end{tikzpicture}~\begin{tikzpicture}
\trU{0}{1}{a}{}{}{}
\trD{-1}{0}{d}{}{}{}
\trD{-1}{1}{b}{}{}{}
\trU{-1}{1}{c}{}{}{}
\end{tikzpicture} &  &  & \begin{tikzpicture}
\trU{-1}{1}{a}{}{}{}
\trD{0}{0}{d}{}{}{}
\trD{-1}{1}{b}{}{}{}
\trU{0}{1}{c}{}{}{}
\end{tikzpicture}~\begin{tikzpicture}
\trU{0}{0}{b}{}{}{}
\trD{-1}{0}{c}{}{}{}
\trD{0}{0}{a}{}{}{}
\trU{-1}{1}{d}{}{}{}
\end{tikzpicture}\tabularnewline
\xymatrix@=4mm{
a & b \ar@{<->}[r] \ar@{.>}^c[l] & c \ar@{<->}[r]  & d
} &  &  & \xymatrix@=4mm{
a \ar@{<->}[r] & b \ar@<-2pt>@{<.}_d[r] \ar@<2pt>@{.>}^a[r] & c \ar@{<->}[r]  & d
}\tabularnewline
\end{tabular}}\hfill{}

\caption{(a) A $(2,4)$ priority strategy for the breaker for $T_{4,2}$. (b)
Dependency digraph of the cells in the orientations of $T_{4,2}$. }
\end{figure}
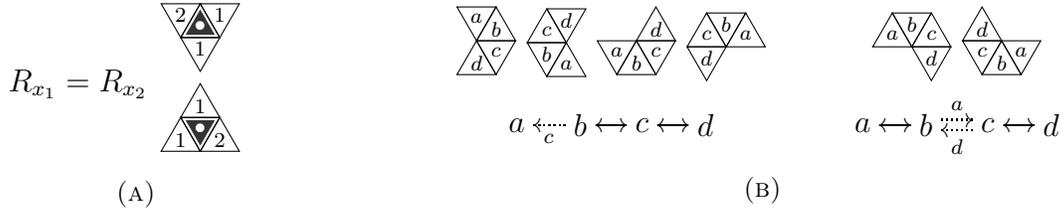

\begin{prop}
The threshold sequence of $T_{4,2}$ is $\tau(T_{4,2})=(0,3,8,\infty)$.\end{prop}
\begin{proof}
The breaker strategy based on the double tiling $\mathsf{T}_{1,1}$
shown in Figure~\ref{fig:TPavings} makes $T_{4,2}$ a $(1,1)$-loser.
It is easy to see that $T_{4,2}$ is a $(3,8)$-winner by Proposition~\ref{pro:twostep}
and a $(3,9)$-loser by Proposition~\ref{pro:surround}. The proof
sequence of Figure~\ref{fig:23PrSeT42} shows that $T_{4,2}$ is
a $(2,3)$-winner. 

It remains to show that $T_{4,2}$ is a $(2,4)$-loser. We are going
to show that the breaker wins following the priority strategy determined
by Figure~\ref{fig:24T42A}. Figure~\ref{fig:24T42B} shows the
possible orientations of the goal animal. In all of these orientations
the dependency digraph of the cells show that cells $b$, $c$ and
$d$ must be marked in the same turn by the maker to achieve the goal
animal. The maker is not able to do so since he only has two marks
in a turn. 
\end{proof}
\begin{figure}
\renewcommand\trrad{.3}

\hfill{}\subfloat[\label{fig:24T43A}]{$R_{x_{1}}=R_{x_{2}}$%
\begin{tabular}{c}
\begin{tikzpicture}
\trUMC{0}{1}{\bullet}
\trD{-1}{1}{1}{}{}{}
\trD{0}{1}{1}{}{}{}
\trD{0}{0}{2}{}{}{}
\end{tikzpicture}\begin{tikzpicture}
\trDMC{0}{0}{\bullet}
\trU{0}{0}{1}{}{}{}
\trU{1}{0}{1}{}{}{}
\trU{0}{1}{2}{}{}{}
\end{tikzpicture}\tabularnewline
\end{tabular}

}\hfill{}\subfloat[\label{fig:24T43B}]{\begin{tikzpicture}
\trU{0}{1}{b}{}{}{}
\trD{-1}{1}{a}{}{}{}
\trD{0}{1}{c}{}{}{}
\trD{0}{0}{d}{}{}{}
\end{tikzpicture}\begin{tikzpicture}
\trD{0}{0}{b}{}{}{}
\trU{0}{0}{a}{}{}{}
\trU{1}{0}{c}{}{}{}
\trU{0}{1}{d}{}{}{}
\end{tikzpicture}$\quad$\lower-0.9cm\hbox{$\xymatrix@=4mm{& d \\ a \ar@{<->}[r]  & b \ar@{<->}[r] \ar@{.>}[u]^a_c & c }$}}\hfill{}

\caption{(a) A $(2,4)$ priority strategy for the breaker for $T_{4,3}$. (b)
Dependency digraph of the cells in the orientations of $T_{4,3}$. }
\end{figure}
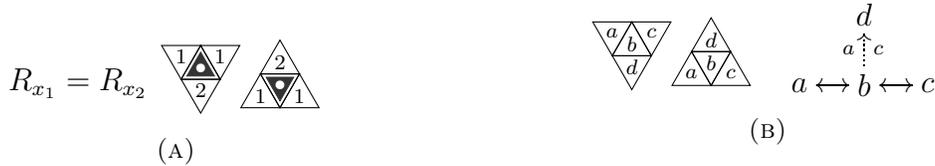

\begin{figure}
\renewcommand\trrad{.45}

\begin{tabular}{c}
\begin{tikzpicture} 
\trUMC{2}{0}{{\bf 1}}
\trUMC{1}{1}{{\bf 2}}
\trDBC{1}{0}{2}
\trUBC{1}{0}{2}
\trUMC{2}{1}{{\bf 3}}
\trUBC{1}{2}{3}
\trDBC{1}{1}{3}
\trDBC{2}{0}{3'}
\trUBC{3}{0}{3'}
\end{tikzpicture}\tabularnewline
\end{tabular}

\caption{\label{fig:23PrSeT43}A forced game play of the $(1\a2,1)$ game for
the animal $T_{4,3}$. }
\end{figure}
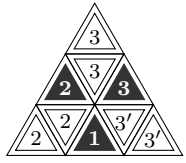

\begin{prop}
The threshold sequence of $T_{4,3}$ is $\tau(T_{4,3})=(0,3,8,\infty)$.\end{prop}
\begin{proof}
The breaker strategy based on the double tiling $\mathsf{T}_{1,1}$
shown in Figure~\ref{fig:TPavings} makes $T_{4,3}$ a $(1,1)$-loser.
It is easy to see that $T_{4,3}$ is a $(3,8)$-winner by Proposition~\ref{pro:twostep}
and a $(3,9)$-loser by Proposition~\ref{pro:surround}. 

We show that the breaker wins the $(2,4)$ game following the priority
strategy determined by Figure~\ref{fig:24T43A}. Figure~\ref{fig:24T43B}
shows the two possible orientations of the goal animal. For both of
these orientations the dependency digraph of the cells show that cells
$a$, $b$ and $c$ must be marked in the same turn by the maker to
achieve the goal animal. The maker is not able to do so since he only
has two marks in a turn. 

It remains to shows that $T_{4,3}$ is a $(2,3)$-winner. We are going
verify that $T_{4,3}$ is a $(1\a2,1)$-winner and use Corollary~\ref{cor:reduce}.
No matter how the breaker picks her first mark, the maker can force
some rotation of the game shown in Figure~\ref{fig:23PrSeT43}. The
breaker has to mark one of the cells labeled with 2 in the second
turn otherwise the maker wins in the third turn. The breaker has to
mark one of the cells labeled with 3 and one of the cells labeled
with 3' in the third turn. This is impossible so the maker wins in
the fourth turn. 
\end{proof}

\section{\label{sec:Polyominoes}Polyominoes}

\selectlanguage{english}%
\usetikzlibrary{patterns}
\newcommand\sq[2]{
 \draw (#1-0.5,#2-0.5) rectangle +(1,1);
}

\newcommand\sqt[2]{
 \draw[line width=1.1pt] (#1-0.5,#2-0.5) rectangle +(1,1);
}

\newcommand\nmaker[3]{
 \draw (#1-0.2,#2-0.2) -- (#1+0.2,#2+0.2); 
 \draw (#1-0.2,#2+0.2) -- (#1+0.2,#2-0.2);
 \draw (#1,#2+0.2) -- (#1,#2-0.2);
 \draw (#1+0.2,#2) -- (#1-0.2,#2);
 \node at (#1+0.35,#2-0.35) {$\scriptstyle #3$};
 \sq{#1}{#2}
}

\newcommand\shadeeven[2]{
 \path[pattern color=black!30,pattern=north east lines] (#1-0.5,#2-0.5) rectangle +(1,1);
}

\newcommand\maker[3]{
 \draw[fill,color=black!80] (#1-0.35,#2-0.35) rectangle +(0.7,0.7);
 \node[color=white] at (#1,#2) {$\scriptstyle #3$};
 \sq{#1}{#2}
}

\newcommand\obreaker[2]{
 \draw (#1-0.35,#2-0.35) rectangle (#1+0.35,#2+0.35); 
 \sq{#1}{#2}
}

\newcommand\breaker[3]{
 \draw (#1-0.35,#2-0.35) rectangle (#1+0.35,#2+0.35); 
 \node at (#1,#2) {$\scriptstyle #3$};
 \sq{#1}{#2} 
}

\newcommand\cell[3]{
 \node at (#1+0.35,#2-0.35) {$\scriptstyle #3$};
 \sq{#1}{#2} 
}

\newcommand\midcell[3]{
\node at (#1, #2) {$\scriptstyle #3$};
\sq{#1}{#2}; }

\newcommand\cellm[3]{
 \node (A) at (#1, #2) {#3};
 \sq{#1}{#2}
}

\newcommand\ccell[6]{  
\node at (#1-0.25, #2+0.25) {$\scriptstyle #3$};  
\node at (#1+0.25, #2+0.25) {$\scriptstyle #4$};  
\node at (#1+0.25, #2-0.25) {$\scriptstyle #5$};  
\node at (#1-0.25, #2-0.25) {$\scriptstyle #6$};  
\sq{#1}{#2}; } 

\selectlanguage{american}%

\begin{figure}
\hfill{}\subfloat[\label{fig:winners}]{%
\begin{tabular}{ccc}
\begin{tikzpicture}[scale=.4]
\cell{1}{1}{};
\cell{2}{1}{};
\cell{3}{1}{};
\cell{4}{1}{};
\cell{4}{2}{};
\end{tikzpicture} & \begin{tikzpicture}[scale=.4]
\cell{1}{1}{};
\cell{2}{1}{};
\cell{3}{1}{};
\cell{4}{1}{};
\cell{3}{2}{};
\end{tikzpicture} & \begin{tikzpicture}[scale=.4]
\cell{1}{1}{};
\cell{2}{1}{};
\cell{3}{1}{};
\cell{3}{2}{};
\cell{4}{2}{};
\end{tikzpicture}\tabularnewline
$L$ & $Y$ & $Z$\tabularnewline
\end{tabular}

}\hfill{}\subfloat[\label{fig:Snaky}]{%
\begin{tabular}{c}
\begin{tikzpicture}[scale=.4]
\cell{1}{1}{};
\cell{2}{1}{};
\cell{3}{1}{};
\cell{4}{1}{};
\cell{4}{2}{};
\cell{5}{2}{};
\end{tikzpicture}\tabularnewline
Snaky\tabularnewline
\end{tabular}}\hfill{}

\caption{(a) Unbiased winners. (b) The only undecided polyomino.}
\end{figure}
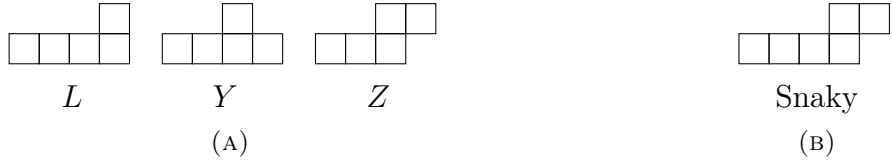

\begin{figure}
\begin{tabular}{c|cccccc}
 & \begin{tikzpicture}[scale=.4,]
\cell{1}{1}{};
\end{tikzpicture} &  & \begin{tikzpicture}[scale=.4,]
\cell{1}{1}{};
\cell{1}{2}{};
\end{tikzpicture} &  & \begin{tikzpicture}[scale=.4,]
\cell{1}{1}{};
\cell{1}{2}{};
\cell{1}{3}{};
\end{tikzpicture} & \begin{tikzpicture}[scale=.4,]
\cell{1}{1}{};
\cell{1}{2}{};
\cell{2}{2}{};
\end{tikzpicture}\tabularnewline
$A$ & $P_{1,1}$ &  & $P_{2,1}$ &  & $P_{3,1}$ & $P_{3,2}$\tabularnewline
\hline 
$\tau(A)$ & $(\infty)$ &  & $(3,\infty)$ &  & $(1,7,\infty)$ & $(1,7,\infty)$\tabularnewline
\end{tabular}

\bigskip{}

\begin{tabular}{c|ccccc}
 & \begin{tikzpicture}[scale=.4]
\cell{1}{1}{};
\cell{1}{2}{};
\cell{1}{3}{};
\cell{1}{4}{};
\end{tikzpicture} & \begin{tikzpicture}[scale=.4]
\cell{2}{1}{};
\cell{1}{1}{};
\cell{1}{2}{};
\cell{1}{3}{};
\end{tikzpicture} & \begin{tikzpicture}[scale=.4]
\cell{1}{1}{};
\cell{1}{2}{};
\cell{2}{2}{};
\cell{1}{3}{};
\end{tikzpicture} & \begin{tikzpicture}[scale=.4]
\cell{1}{1}{};
\cell{1}{2}{};
\cell{2}{1}{};
\cell{2}{2}{};
\end{tikzpicture} & \begin{tikzpicture}[scale=.4]
\cell{1}{1}{};
\cell{1}{2}{};
\cell{2}{3}{};
\cell{2}{2}{};
\end{tikzpicture}\tabularnewline
$A$ & $P_{4,1}$ & $P_{4,2}$ & $P_{4,3}$ & $P_{4,4}$ & $P_{4,5}$\tabularnewline
\hline 
$\tau(A)$ & $(1,3,11,\infty)$ & $(1,5,11,\infty)$ & $(1,3,11,\infty)$ & $(0,3,5,\infty)$ & $(1,3,11,\infty)$\tabularnewline
\end{tabular}\caption{\label{fig:polyominoes}Polyominoes and their threshold sequences
up to size four.}
\end{figure}

\begin{figure}
\begin{tabular}{ccccc}
\begin{tikzpicture}[scale=0.4,line cap=round,line join=round,>=triangle 45,x=1.0cm,y=1.0cm]
\draw (0,0)-- (2,0);
\draw (0,1)-- (2,1);
\draw (0,2)-- (2,2);
\draw (0,3)-- (2,3);
\draw (0,3)-- (0,0);
\draw (1,0)-- (1,3);
\draw (2,3)-- (2,0);
\draw [line width=1.6pt,color=ffqqww] (0.5,0.3)-- (0.5,-0.3);
\draw [line width=1.6pt,color=ffqqww] (1.5,1.3)-- (1.5,0.7);
\draw [line width=1.6pt,color=ffqqww] (0.5,2.3)-- (0.5,1.7);
\draw [line width=1.6pt,color=ffqqww] (1.5,3.3)-- (1.5,2.7);
\end{tikzpicture} &  & \begin{tikzpicture}[scale=0.4,line cap=round,line join=round,>=triangle 45,x=1.0cm,y=1.0cm]
\draw (0,0)-- (2,0);
\draw (0,1)-- (2,1);
\draw (0,2)-- (2,2);
\draw (0,3)-- (2,3);
\draw (0,3)-- (0,0);
\draw (1,0)-- (1,3);
\draw (2,3)-- (2,0);
\draw [line width=1.6pt,color=ffqqww] (0.5,0.3)-- (0.5,-0.3);
\draw [line width=1.6pt,color=ffqqww] (1.5,0.3)-- (1.5,-0.3);
\draw [line width=1.6pt,color=ffqqww] (1.5,1.3)-- (1.5,0.7);
\draw [line width=1.6pt,color=ffqqww] (1.5,2.3)-- (1.5,1.7);
\draw [line width=1.6pt,color=ffqqww] (0.5,1.3)-- (0.5,0.7);
\draw [line width=1.6pt,color=ffqqww] (0.5,2.3)-- (0.5,1.7);
\draw [line width=1.6pt,color=ffqqww] (0.5,3.3)-- (0.5,2.7);
\draw [line width=1.6pt,color=ffqqww] (1.5,3.3)-- (1.5,2.7);
\end{tikzpicture} &  & \begin{tikzpicture}[scale=0.4,line cap=round,line join=round,>=triangle 45,x=1.0cm,y=1.0cm]
\draw (0,0)-- (4,0);
\draw (0,1)-- (4,1);
\draw (0,2)-- (4,2);
\draw (0,3)-- (4,3);
\draw (0,4)-- (4,4);
\draw (0,0)-- (0,4);
\draw (1,0)-- (1,4);
\draw (2,0)-- (2,4);
\draw (3,0)-- (3,4);
\draw (4,0)-- (4,4);
\draw [line width=1.6pt,color=ffqqww] (0.5,0.3)-- (0.5,-0.3);
\draw [line width=1.6pt,color=ffqqww] (1.5,0.3)-- (1.5,-0.3);
\draw [line width=1.6pt,color=ffqqww] (2.5,0.3)-- (2.5,-0.3);
\draw [line width=1.6pt,color=ffqqww] (3.5,0.3)-- (3.5,-0.3);
\draw [line width=1.6pt,color=ffqqww] (1.5,2.3)-- (1.5,1.7);
\draw [line width=1.6pt,color=ffqqww] (0.5,2.3)-- (0.5,1.7);
\draw [line width=1.6pt,color=ffqqww] (2.5,2.3)-- (2.5,1.7);
\draw [line width=1.6pt,color=ffqqww] (3.5,1.7)-- (3.5,2.3);
\draw [line width=1.6pt,color=ffqqww] (1.5,4.3)-- (1.5,3.7);
\draw [line width=1.6pt,color=ffqqww] (0.5,4.3)-- (0.5,3.7);
\draw [line width=1.6pt,color=ffqqww] (2.5,4.3)-- (2.5,3.7);
\draw [line width=1.6pt,color=ffqqww] (3.5,4.3)-- (3.5,3.7);

\draw [line width=1.6pt,color=ffqqww] (-0.3,0.5)-- (0.3,0.5);
\draw [line width=1.6pt,color=ffqqww] (-0.3,1.5)-- (0.3,1.5);
\draw [line width=1.6pt,color=ffqqww] (-0.3,2.5)-- (0.3,2.5);
\draw [line width=1.6pt,color=ffqqww] (-0.3,3.5)-- (0.3,3.5);
\draw [line width=1.6pt,color=ffqqww] (1.7,0.5)-- (2.3,0.5);
\draw [line width=1.6pt,color=ffqqww] (1.7,1.5)-- (2.3,1.5);
\draw [line width=1.6pt,color=ffqqww] (1.7,2.5)-- (2.3,2.5);
\draw [line width=1.6pt,color=ffqqww] (1.7,3.5)-- (2.3,3.5);
\draw [line width=1.6pt,color=ffqqww] (3.7,0.5)-- (4.3,0.5);
\draw [line width=1.6pt,color=ffqqww] (3.7,1.5)-- (4.3,1.5);
\draw [line width=1.6pt,color=ffqqww] (3.7,2.5)-- (4.3,2.5);
\draw [line width=1.6pt,color=ffqqww] (3.7,3.5)-- (4.3,3.5);

\end{tikzpicture}\tabularnewline
$\mathsf{T}_{1,1}$ &  & $\mathsf{T}_{2,1}$ &  & $\mathsf{T}_{2,2}$\tabularnewline
\end{tabular}\foreignlanguage{english}{}

\caption{\label{fig:Pavings}Rectangular pavings.}
\end{figure}
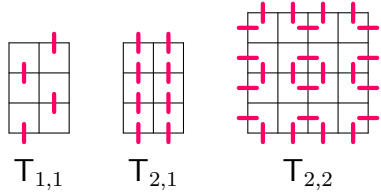

Polyomino achievement games were invented by Harary. The first proof
sequences for the unbiased games appeared in \cite{bode.harborth:hexagonal}.
Every known $(1,1)$-winner is a subset of one of the winning polyominoes
shown in Figure~\ref{fig:winners}. The only undecided \cite{AchievingSnaky,harary:is,harborth.seemann:snaky*1,harborth.seemann:snaky,sieben:snaky,prooftree}
polyomino Snaky, shown in Figure~\ref{fig:Snaky}, is conjectured
\cite{ww3} to be a winner. Biased $(1,2)$ polyomino set games were
studied in \cite{sieben:fisher,sieben:wild}.

In this section, we find the threshold sequences of polyominoes up
to size four. The results are summarized in Figure~\ref{fig:polyominoes}.
The proof of the following result is an easy exercise.
\begin{thm}
The threshold sequences for $P_{1,1}$ and $P_{2,1}$ are $\tau(P_{1,1})=(\infty)$
and $\tau(P_{2,1})=(3,\infty)$.
\end{thm}

\begin{thm}
The threshold sequence for $P_{3,1}$ and $P_{3,2}$ is $\tau(P_{3,1})=\tau(P_{3,2})=(1,7,\infty)$. \end{thm}
\begin{proof}
Polyominoes $P_{3,1}$ and $P_{3,2}$ are $(1,1)$-winners since they
are subsets of the winner $L$ shown in Figure~\ref{fig:winners}.
The breaker wins the $(1,2)$ games for $P_{3,1}$ and $P_{3,2}$
using the strategy based on the double paving $\mathsf{T}_{2,2}$
and $\mathsf{T}_{2,1}$ respectively. It is easy to see that $P_{3,1}$
and $P_{3,2}$ are $(2,7)$-winners using Proposition~\ref{pro:twostep}
and $(2,8)$-losers using Proposition~\ref{pro:surround}.
\end{proof}
Now we consider the size 4 animals.
\begin{lem}
Animals $P_{4,1}$, $P_{4,2}$, $P_{4,3}$ and $P_{4,5}$ are $(3,11)$-winners
and $(3,12)$-losers.\end{lem}
\begin{proof}
The result follows easily from Proposition~\ref{pro:twostep} and
Proposition~\ref{pro:surround}.
\end{proof}
Note that $P_{4,4}$ is missing from the lemma since Proposition~\ref{pro:twostep}
does not apply for this animal.

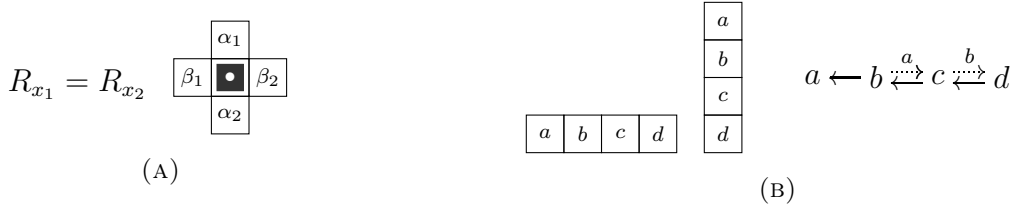
\begin{figure}
\hfill{}\subfloat[\label{fig:24P41A}]{$R_{x_1}=R_{x_2}$
\begin{tabular}{c}
\begin{tikzpicture} [scale=0.5]
\midcell{2}{3}{\alpha_1}
\midcell{1}{2}{\beta_1}
\midcell{2}{1}{\alpha_2}
\midcell{3}{2}{\beta_2}
\maker{2}{2}{\bullet}
\end{tikzpicture}\tabularnewline
\end{tabular}

}\hfill{}\subfloat[\label{fig:24P41B}]{%
\begin{tabular}{cccc}
\begin{tikzpicture} [scale=0.5]
\midcell{1}{1}{a}
\midcell{2}{1}{b}
\midcell{3}{1}{c}
\midcell{4}{1}{d}
\end{tikzpicture} & \begin{tikzpicture} [scale=0.5]
\midcell{1}{1}{d}
\midcell{1}{2}{c}
\midcell{1}{3}{b}
\midcell{1}{4}{a}
\end{tikzpicture} &  & \lower-0.9cm\hbox{$\xymatrix@=4mm{
a  & b \ar[l] \ar@<2pt>@{.>}[r] \ar@<2pt>@{.>}[r]^{a} & c \ar@<2pt>[l] \ar@<2pt>@{.>}[r]^{b} & d \ar@<2pt>[l]   
}$}\tabularnewline
\end{tabular}

}\hfill{}

\caption{(a) A $(2,4)$ priority strategy for the breaker for $P_{4,1}$ (b)
Dependency digraph of the cells in the orientations of $P_{4,1}$. }
\end{figure}

\begin{prop}
The threshold sequence of $P_{4,1}$ is $\tau(P_{4,1})=(1,3,11,\infty)$.\end{prop}
\begin{proof}
Polyomino $P_{4,1}$ is a subset of the winner $L$ shown in Figure~\ref{fig:winners}
so the maker wins the $(1,1)$ game and therefore the $(2,3)$ game
by Corollary~\ref{cor:reduce}. 

We show that the breaker wins the $(2,4)$ game following the priority
strategy determined by Figure~\ref{fig:24P41A}. The breaker uses
priorities $(\alpha_{1},\alpha_{2})$ during her first mark in the
response set and priorities $(\beta_{1},\beta_{2})$ during her second
mark in the response set. 

Figure~\ref{fig:24P41B} shows the dependency digraph of the cells
of the goal animal in both orientations. It is clear from the digraph
that cells $b$, $c$ and $d$ must be marked in the same turn by
the maker to achieve the goal animal. The maker is not able to do
so since he only has two marks in a turn.
\end{proof}
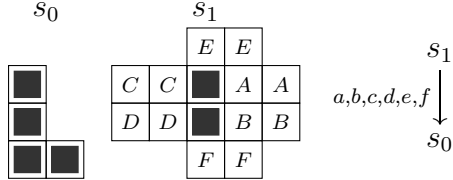
\begin{figure}
\begin{tabular}{cc}
$s_{0}$ & $s_{1}$\tabularnewline
\begin{tikzpicture} [scale=0.5]
\maker{1}{1}{}
\maker{1}{2}{}
\maker{1}{3}{}
\maker{2}{1}{}
\end{tikzpicture} & \begin{tikzpicture} [scale=0.5]
\maker{3}{2}{}
\maker{3}{3}{}
\midcell{1}{2}{D}
\midcell{1}{3}{C}
\midcell{2}{2}{D}
\midcell{2}{3}{C}
\midcell{3}{1}{F}
\midcell{3}{4}{E}
\midcell{4}{1}{F}
\midcell{4}{2}{B}
\midcell{4}{3}{A}
\midcell{4}{4}{E}
\midcell{5}{2}{B}
\midcell{5}{3}{A}
\end{tikzpicture}\tabularnewline
\end{tabular}\lower-0.5cm\hbox{
$\xymatrix@=7mm{
s_1 \ar[d]_{a,b,c,d,e,f} \\
s_0 \\
}$
}

\caption{\label{fig:P_4_2 w (2,5)}A proof sequence of the maker strategy for
$(2,5)$-achieving $P_{4,2}$. }
\end{figure}

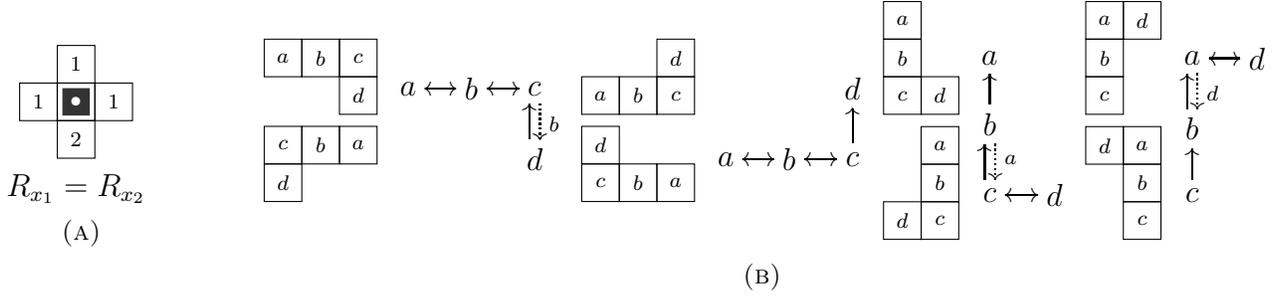
\begin{figure}
\hfill{}\subfloat[\label{fig:26P42A}]{%
\begin{tabular}{c}
\begin{tikzpicture} [scale=0.5]
\midcell{2}{3}{1}
\midcell{1}{2}{1}
\midcell{2}{1}{2}
\midcell{3}{2}{1}
\maker{2}{2}{\bullet}
\end{tikzpicture}\tabularnewline
$R_{x_1}=R_{x_2}$
\tabularnewline
\end{tabular}

}\hfill{}\subfloat[\label{fig:26P42B}]{

\begin{tabular}{c}
\begin{tikzpicture} [scale=0.5]
\midcell{1}{1}{a}
\midcell{2}{1}{b}
\midcell{3}{1}{c}
\midcell{3}{0}{d}
\end{tikzpicture}\tabularnewline
\begin{tikzpicture} [scale=0.5]
\midcell{1}{1}{c}
\midcell{2}{1}{b}
\midcell{3}{1}{a}
\midcell{1}{0}{d}
\end{tikzpicture}\tabularnewline
\end{tabular}\lower-0.5cm\hbox{$\xymatrix@=4mm{
a \ar@{<->}[r] & b \ar@{<->}[r]  & c \ar@<2pt>@{.>}[d]^b \\ 
               &                 & d \ar@<2pt>[u] }$}%
\begin{tabular}{c}
\begin{tikzpicture} [scale=0.5]
\midcell{1}{0}{a}
\midcell{2}{0}{b}
\midcell{3}{0}{c}
\midcell{3}{1}{d}
\end{tikzpicture}\tabularnewline
\begin{tikzpicture} [scale=0.5]
\midcell{1}{0}{c}
\midcell{2}{0}{b}
\midcell{3}{0}{a}
\midcell{1}{1}{d}
\end{tikzpicture}\tabularnewline
\end{tabular}\lower-0.5cm\hbox{$\xymatrix@=4mm{
               &                 & d  \\
a \ar@{<->}[r] & b \ar@{<->}[r]  & c \ar[u] 
}$}%
\begin{tabular}{c}
\begin{tikzpicture} [scale=0.5]
\midcell{1}{1}{c}
\midcell{1}{2}{b}
\midcell{1}{3}{a}
\midcell{2}{1}{d}
\end{tikzpicture}\tabularnewline
\begin{tikzpicture} [scale=0.5]
\midcell{1}{1}{c}
\midcell{1}{2}{b}
\midcell{1}{3}{a}
\midcell{0}{1}{d}
\end{tikzpicture}\tabularnewline
\end{tabular}\lower-0.9cm\hbox{$\xymatrix@=4mm{
a \\
b \ar[u] \ar@<2pt>@{.>}[d]^a \\
c \ar@<2pt>[u] \ar@{<->}[r]  & d  \\ 
}$}%
\begin{tabular}{c}
\begin{tikzpicture} [scale=0.5]
\midcell{1}{1}{c}
\midcell{1}{2}{b}
\midcell{1}{3}{a}
\midcell{2}{3}{d}
\end{tikzpicture}\tabularnewline
\begin{tikzpicture} [scale=0.5]
\midcell{1}{1}{c}
\midcell{1}{2}{b}
\midcell{1}{3}{a}
\midcell{0}{3}{d}
\end{tikzpicture}\tabularnewline
\end{tabular}\lower-0.9cm\hbox{$\xymatrix@=4mm{
a \ar@{<->}[r] \ar@<2pt>@{.>}[d]^d  & d \\
b \ar@<2pt>[u]  \\
c \ar[u]   \\ 
}$}

}\hfill{}

\caption{(a) A $(2,6)$ priority strategy for the breaker for $P_{4,2}$ (b)
Dependency digraph of the cells in the orientations of $P_{4,2}$
with secondary arrows omitted. }
\end{figure}

\begin{prop}
The threshold sequence of $P_{4,2}$ is $\tau(P_{4,2})=(1,5,11,\infty)$.\end{prop}
\begin{proof}
Polyomino $P_{4,2}$ is a subset of the winner $L$ shown in Figure~\ref{fig:winners}
so maker wins the $(1,1)$ game. The breaker wins the $(1,2)$ game
since $P_{4,2}$ contains the $(1,2)$-loser $P_{3,1}$. 

The maker wins the $(2,5)$ game using the proof sequence of Figure~\ref{fig:P_4_2 w (2,5)}.
It is easy to see that the breaker wins the $(2,6)$ game following
the priority strategy determined by Figure~\ref{fig:26P42A}. 
\end{proof}
Note that a somewhat more complicated proof sequence shows that $P_{4,2}$
is a $(1\a2,2)$-winner which also implies that $P_{4,2}$ is a $(2,5)$-winner.

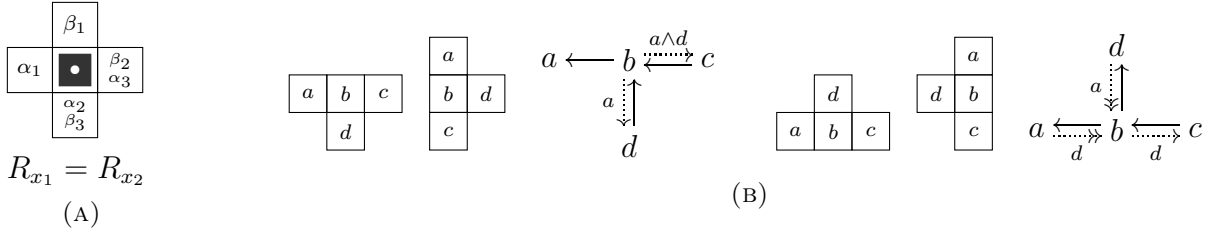
\begin{figure}
\hfill{}\subfloat[\label{fig:26P43A}]{%
\begin{tabular}{c}
\begin{tikzpicture} [scale=0.6]
\midcell{2}{3}{\beta_1}
\midcell{1}{2}{\alpha_1}
\midcell{2}{1}{\genfrac{}{}{0pt}{}{\alpha_2}{\beta_3}}
\midcell{3}{2}{\genfrac{}{}{0pt}{}{\beta_2}{\alpha_3}}
\maker{2}{2}{\bullet}
\end{tikzpicture}\tabularnewline
$R_{x_1}=R_{x_2}$
\tabularnewline
\end{tabular}

}\hfill{}\subfloat[\label{fig:24P43B}]{

\begin{tabular}{ccccccc}
\begin{tikzpicture} [scale=0.5]
\midcell{1}{1}{a}
\midcell{2}{1}{b}
\midcell{3}{1}{c}
\midcell{2}{0}{d}
\end{tikzpicture} & \begin{tikzpicture} [scale=0.5]
\midcell{1}{1}{c}
\midcell{1}{2}{b}
\midcell{1}{3}{a}
\midcell{2}{2}{d}
\end{tikzpicture} & \lower-1.1cm\hbox{$\xymatrix@=6mm{
a & b \ar[l] \ar@<-2pt>@{.>}[d]_{a} \ar@<2pt>@{.>}[r]^{a\land d} & c \ar@<2pt>[l]  \\ 
               & d \ar@<-2pt>[u]    &  }$} &  & \begin{tikzpicture} [scale=0.5]
\midcell{1}{0}{a}
\midcell{2}{0}{b}
\midcell{3}{0}{c}
\midcell{2}{1}{d}
\end{tikzpicture} & \begin{tikzpicture} [scale=0.5]
\midcell{1}{1}{c}
\midcell{1}{2}{b}
\midcell{1}{3}{a}
\midcell{0}{2}{d}
\end{tikzpicture} & \lower-1.3cm\hbox{$\xymatrix@=6mm{
                         & d \ar@{.>>}@<-2pt>[d]_a               &   \\
a \ar@{.>>}@<-2pt>[r]_d  & b \ar@<-2pt>[l] \ar@<-2pt>[u] \ar@{.>}@<-2pt>[r]_d  & c \ar@<-2pt>[l]  
}$}\tabularnewline
\end{tabular}

}\hfill{}

\caption{(a) A $(2,4)$ priority strategy for the breaker for $P_{4,3}$ (b)
Dependency digraph of the cells in the orientations of $P_{4,3}$. }
\end{figure}

\begin{prop}
The threshold sequence of $P_{4,3}$ is $\tau(P_{4,3})=(1,3,11,\infty)$.\end{prop}
\begin{proof}
Polyomino $P_{4,3}$ is a subset of the winner $Y$ shown in Figure~\ref{fig:winners}
so the maker wins the $(1,1)$ game and therefore the $(2,3)$ game.
The breaker wins the $(1,2)$ game since $P_{4,3}$ contains the $(1,2)$-loser
$P_{3,1}$. 

We show that the breaker wins the $(2,4)$ game following the priority
strategy determined by Figure~\ref{fig:26P43A}. The breaker uses
priorities $(\alpha_{1},\alpha_{2},\alpha_{3})$ during her first
mark in a response set and priorities $(\beta_{1},\beta_{2},\beta_{3})$
during her second mark in the response set. 

Figure~\ref{fig:24P43B} shows the dependency digraph of the cells
of the goal animal in all orientations. In the first two orientations,
we have $b\sim c\sim d$ so cells $b$, $c$ and $d$ must be marked
in the same turn by the maker to achieve the goal animal. In the last
two orientations, either cells $a$, $b$ and $c$ or cells $d$,
$b$ and $c$ must be marked in the same turn as explained in Example~\ref{exa:secondary}.
The maker is not able to do so since he only has two marks in a turn.
\end{proof}
\begin{figure}
\begin{tabular}{ccc}
$s_{0}$ & $s_{1}$ & $s_{2}$\tabularnewline
\begin{tikzpicture} [scale=0.5]
\maker{1}{1}{}
\maker{1}{2}{}
\maker{2}{1}{}
\maker{2}{2}{}
\end{tikzpicture} & \begin{tikzpicture} [scale=0.5]
\maker{1}{2}{}
\maker{2}{2}{}
\midcell{1}{3}{A}
\midcell{2}{3}{A}
\midcell{1}{1}{B}
\midcell{2}{1}{B}
\end{tikzpicture} & \lower .8mm\hbox{
\begin{tikzpicture} [scale=0.61]
\maker{2}{2}{}
\ccell{1}{1}{a}{}{c}{}
\ccell{1}{2}{A}{}{c}{d}
\ccell{1}{3}{a}{}{}{d}
\ccell{2}{1}{a}{b}{C}{}
\ccell{2}{3}{a}{b}{}{D}
\ccell{3}{1}{}{b}{c}{}
\ccell{3}{2}{}{B}{c}{d}
\ccell{3}{3}{}{b}{}{d}
\end{tikzpicture}
}\tabularnewline
\end{tabular}\lower-0.9cm\hbox{
$\xymatrix@=5mm{
s_2 \ar[d]_{a,b,c,d} \\
s_1 \ar[d]_{a,b}  \\
s_0 \\
}$
}

\caption{\label{fig:P_4_4 w (2,3)}A proof sequence of the maker strategy for
$(1\a2,1)$-achieving $P_{4,4}$. }
\end{figure}
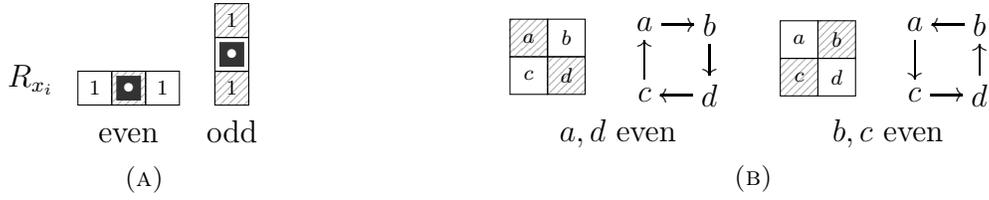
\begin{figure}
\hfill{}\subfloat[\label{fig:24P44A}]{$R_{x_i}$

\begin{tabular}{cc}
\begin{tikzpicture} [scale=0.45]
\midcell{1}{2}{1}
\midcell{3}{2}{1}
\shadeeven{2}{2}
\maker{2}{2}{\bullet}
\end{tikzpicture} & \begin{tikzpicture} [scale=0.45]
\shadeeven{2}{3}
\shadeeven{2}{1}
\midcell{2}{3}{1}
\midcell{2}{1}{1}
\maker{2}{2}{\bullet}
\end{tikzpicture}\tabularnewline
even & odd\tabularnewline
\end{tabular}

}\hfill{}\subfloat[\label{fig:24P44B}]{%
\begin{tabular}{cc}
\begin{tikzpicture} [scale=0.5]
\midcell{1}{1}{c}{0}{}{}
\shadeeven{1}{2}
\midcell{1}{2}{a}
\shadeeven{2}{1}
\midcell{2}{1}{d}
\midcell{2}{2}{b}
\end{tikzpicture}
$\quad$\lower-0.83cm\hbox{$\xymatrix@=4mm{
a \ar[r]  & b \ar[d] \\ 
c \ar[u]  & d \ar[l] }$} & \tabularnewline
$a,d$ even & \tabularnewline
\end{tabular}%
\begin{tabular}{c}
\begin{tikzpicture} [scale=0.5]
\shadeeven{1}{1}
\midcell{1}{1}{c}
\midcell{1}{2}{a}
\midcell{2}{1}{d}
\shadeeven{2}{2}
\midcell{2}{2}{b}
\end{tikzpicture}
$\quad$\lower-0.85cm\hbox{$\xymatrix@=4mm{
a \ar[d]  & b \ar[l] \\ 
c \ar[r]  & d \ar[u] }$}\tabularnewline
$b,c$ even\tabularnewline
\end{tabular}}\hfill{}

\caption{(a) A $(2,4)$ and $(3,6)$ priority strategy for the breaker for
$P_{4,4}$. The priorities depend on the parity of the current maker
mark. Even cells are shaded. (b) Dependency digraphs of the cells
in the placements of $P_{4,4}$. }
\end{figure}

\begin{prop}
The threshold sequence of $P_{4,4}$ is $\tau(P_{4,4})=(0,3,5,\infty)$. \end{prop}
\begin{proof}
Polyomino $P_{4,4}$ is not a subset of any of the winners shown in
Figure~\ref{fig:winners} so it is a $(1,1)$-loser. Actually, the
breaker wins using the paving strategy based on paving $\mathsf{T}_{1,1}$
shown in Figure~\ref{fig:Pavings}.

The maker wins the $(1\a2,1)$ game using the proof sequence in Figure~\ref{fig:P_4_4 w (2,3)}.
So by Corollary~\ref{cor:reduce} the maker wins the $(2,3)$ game
. The maker also wins the $(1\a3,1)$ and therefore the $(3,5)$ game
since the $(1\a3,1)$ game is easier for the maker than the $(1\a2,1)$
game. 

We show that the breaker wins the $(2,4)$ game and the $(3,6)$ game
following the priority strategy determined by Figure~\ref{fig:24P44A}.
We say that a cell on the board with coordinates $(x,y)$ is even
if $x+y$ is even. Otherwise the cell is called odd. The parities
of the cells therefore form a checkerboard pattern. The priorities
in the breaker strategy depend on the parity of the current cell marked
by the maker. Figure~\ref{fig:24P44B} shows the dependency digraph
of the cells of the goal animal. It is clear from the digraph that
all four cells have to be marked in a single turn. The maker is not
able to do so since he has fewer than four marks in a turn in both
the $(2,4)$ and $(3,6)$ games.\end{proof}
\begin{prop}
The threshold sequence of $P_{4,5}$ is $\tau(P_{4,5})=(1,c,11,\infty)$
where $c\ge3$.\end{prop}
\begin{proof}
Polyomino $P_{4,5}$ is a subset of the winner $Z$ shown in Figure~\ref{fig:winners}
so the maker wins the $(1,1)$ game and therefore the $(2,3)$ game.
The breaker wins the $(1,2)$ game since $P_{4,5}$ contains the $(1,2)$-loser
$P_{3,2}$.
\end{proof}
It remains to show that $P_{4,5}$ is a $(2,4)$-loser. For this we
need a more complicated priority strategy for the breaker.

\section{\label{sec:historyDep}The history dependent priority strategy}

\begin{figure}
\hfill{}\subfloat[\label{fig:historyStrat}]{%
\begin{tabular}{ccccc}
even & \begin{tikzpicture} [scale=0.45]
\midcell{2}{3}{1}
\maker{1}{2}{}
\midcell{2}{1}{1}
\maker{3}{2}{}
\shadeeven{2}{2}
\maker{2}{2}{\bullet}
\end{tikzpicture} & \begin{tikzpicture} [scale=0.45]
\maker{1}{2}{}
\shadeeven{2}{2}
\shadeeven{3}{3}
\shadeeven{3}{1}
\maker{2}{2}{\bullet}
\midcell{3}{3}{1}
\midcell{3}{1}{1}
\midcell{3}{2}{2}
\end{tikzpicture} & \begin{tikzpicture} [scale=0.45]
\maker{3}{2}{}
\shadeeven{2}{2}
\shadeeven{1}{3}
\shadeeven{1}{1}
\midcell{1}{3}{1}
\midcell{1}{1}{1}
\midcell{1}{2}{2}
\maker{2}{2}{\bullet}
\end{tikzpicture} & \raise.45cm\hbox{%
\begin{tikzpicture} [scale=0.45]
\midcell{1}{2}{1}
\midcell{3}{2}{1}
\shadeeven{2}{2}
\maker{2}{2}{\bullet}
\end{tikzpicture}
}\tabularnewline
odd & \begin{tikzpicture} [scale=0.45]
\shadeeven{2}{3}
\shadeeven{1}{2}
\shadeeven{2}{1}
\shadeeven{3}{2}
\maker{2}{3}{}
\midcell{1}{2}{1}
\maker{2}{1}{}
\midcell{3}{2}{1}
\maker{2}{2}{\bullet}
\end{tikzpicture} & \begin{tikzpicture} [scale=0.45]
\shadeeven{2}{1}
\shadeeven{2}{3}
\midcell{1}{3}{1}
\maker{2}{1}{}
\midcell{3}{3}{1}
\midcell{2}{3}{2}
\maker{2}{2}{\bullet}
\end{tikzpicture} & \begin{tikzpicture} [scale=0.45]
\shadeeven{2}{3}
\shadeeven{2}{1}
\midcell{1}{1}{1}
\maker{2}{3}{}
\midcell{3}{1}{1}
\midcell{2}{1}{2}
\maker{2}{2}{\bullet}
\end{tikzpicture} & \begin{tikzpicture} [scale=0.45]
\shadeeven{2}{3}
\shadeeven{2}{1}
\midcell{2}{3}{1}
\midcell{2}{1}{1}
\maker{2}{2}{\bullet}
\end{tikzpicture}\tabularnewline
\end{tabular}

}\hfill{}\subfloat[\label{fig:placements}]{%
\begin{tabular}{c}
\begin{tikzpicture} [scale=0.45]
\shadeeven{1}{1}
\shadeeven{2}{2}
\shadeeven{2}{4}
\shadeeven{1}{3}
\midcell{1}{1}{a}
\midcell{1}{2}{b}
\midcell{2}{2}{c}
\midcell{2}{3}{d}
\midcell{0}{1}{h_1}
\midcell{2}{1}{h_3}
\midcell{3}{2}{h_5}
\midcell{2}{4}{h_4}
\midcell{1}{3}{h_2}
\end{tikzpicture}\tabularnewline
\end{tabular}

}\hfill{}

\caption{(a) A $(2,4)$ history dependent priority breaker strategy for $P_{4,5}$.
The priorities depend on the parity of the current maker mark. Even
cells are shaded. (b) The goal cells $\{a,b,c,d\}$ and the considered
history cells $\{h_{1},\ldots,h_{5}\}$.}
\end{figure}
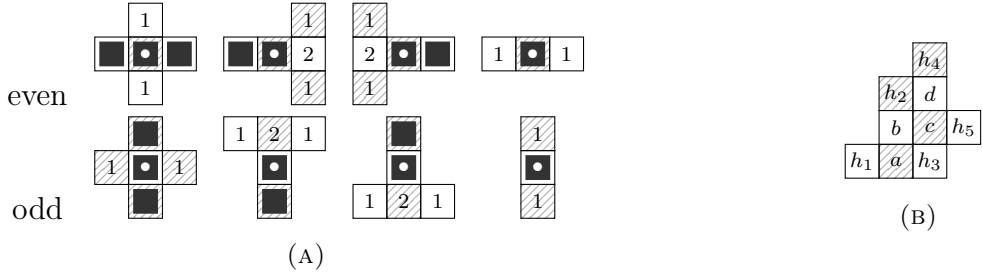

We introduce a priority strategy where the priorities of the cells
depend on the state of certain \emph{history cells}. If some of the
history cells are already marked by the maker, then the breaker uses
different priorities. The priorities do not change if some the history
cells are marked by the breaker. Figure~\ref{fig:historyStrat} shows
a history dependent priority strategy. The priorities depend on the
parity of the current maker mark. The first row shows the priorities
if this parity is even while the second row shows the priorities if
the parity is odd. Both of these rows contain four rules. The breaker
uses the first available rule for which the history cells have the
correct state. 

For example if the current maker mark is even, then the breaker uses
the rules in the first row. The first rule requires that the history
cells located on left and on the right of the current cell are both
marked by the maker. If the condition is satisfied, then the breaker
uses this first rule. If any of the history cells are unmarked or
marked by the breaker, then the breaker jumps to the next rule. This
second rule only requires that the history cell on the left of the
current mark is marked by the maker. It is clear that the requirement
for one of the rules is always satisfied. Note that the conditions
for the second and the third rules are never satisfied together. Swapping
these two rules in either row has no effect on the strategy. 

\begin{algorithm}
\begin{tabular}{ll}
\begin{tabular}{l}
1.\tabularnewline
\end{tabular} & \textbf{for each $G$} placement of the goal animal\tabularnewline
\begin{tabular}{l}
2.\tabularnewline
3.\tabularnewline
4.\tabularnewline
5.\tabularnewline
6.\tabularnewline
7.\tabularnewline
8.\tabularnewline
9.\tabularnewline
10.\tabularnewline
11.\tabularnewline
12.\tabularnewline
13.\tabularnewline
\end{tabular} & $\quad$%
\begin{tabular}{|l}
initialize the set $H$ of history cells \tabularnewline
$E:=G\cup H$\tabularnewline
push $(E,M:=\{\},B:=\{\})$ to $Positions$\tabularnewline
\textbf{while} $Positions$ is not empty\tabularnewline
$\quad$%
\begin{tabular}{|ll}
pop $(E,M,B)$ from $Positions$ & \tabularnewline
\textbf{if} $|G\cap E|\le2$  & \tabularnewline
$\quad$%
\begin{tabular}{|c}
breaker strategy fails, \textbf{stop}\tabularnewline
\end{tabular} & \tabularnewline
\textbf{for each} $F\subseteq E$ such that $|F|\le2$  & \tabularnewline
$\quad$%
\begin{tabular}{|l}
$\tilde{E}:=E\setminus F$\tabularnewline
$\tilde{M}:=M\cup F$\tabularnewline
$nPositions:=$AddBreaker$(G,\tilde{E},\tilde{M},B,F)$\tabularnewline
push $nPositions$ to $Positions$ \tabularnewline
\end{tabular} & \tabularnewline
\end{tabular}\tabularnewline
\end{tabular}\tabularnewline
\begin{tabular}{l}
14.\tabularnewline
\end{tabular} & strategy works\tabularnewline
\end{tabular}

\vspace{3mm}

\dotfill
\begin{enumerate}
\item We need to verify that none of the placements of the goal animal can
be marked by the maker in any order.
\item We collect the possible history cells in $H$. For an even goal cell
we add the cell on the left and on the right of the goal cell. For
an odd goal cell we add the cell above and below the goal cell.
\item The set $E$ of empty cells contains the goal and the history cells.
Although the maker only needs to mark the goal cells, marking the
history cells as well may affect his success. So we test every permutation
of the empty cells.
\item Variable $Positions$ is a stack that contains the set of game positions
that are promising for the maker. A position is determined by the
set $E$ of empty cells, the set $M$ of cells marked by the maker
and the set $B$ of cells marked by the breaker. At the beginning
we only have one position in which every cell is empty.
\item The analysis continues while there are any promising positions left
to consider.
\item We take one of the promising positions for further consideration.
\item If the position has fewer than 3 unmarked goal cells, then the maker
can win since he is allowed to mark 2 cells. 
\item This means the breaker strategy might fail, so we stop. Note that
the breaker strategy might actually work but require a more sophisticated
analysis.
\item We consider each possible one and two element subset $F$ of the set
of remaining empty cells $E$. Set $F$ contains the cells that the
maker is about to mark. We consider 1-element subsets since the maker
may use one of his marks somewhere else on the playing board.
\item We remove the current maker marks form the set of empty cells.
\item We add the current maker marks to the set of maker marks.
\item We call the function in Algorithm~2 to add the defensive breaker
marks corresponding to the current maker marks. The function returns
a possibly empty set of new positions.
\item We add the new positions to the stack of positions.
\item We run out of positions promising for the maker. This means the breaker
strategy worked.
\end{enumerate}
\caption{Analyze strategy}
\end{algorithm}

\begin{algorithm}
\begin{tabular}{ll}
\begin{tabular}{c}
1.\tabularnewline
\end{tabular} & push $(E,M,B)$ to $Positions$\tabularnewline
\begin{tabular}{c}
2.\tabularnewline
\end{tabular} & \textbf{for each} $f\in F$\tabularnewline
\begin{tabular}{l}
3.\tabularnewline
4.\tabularnewline
5.\tabularnewline
6.\tabularnewline
7.\tabularnewline
8.\tabularnewline
9.\tabularnewline
10.\tabularnewline
11.\tabularnewline
12.\tabularnewline
13.\tabularnewline
14.\tabularnewline
15.\tabularnewline
16.\tabularnewline
\end{tabular} & \textbf{$\quad$}%
\begin{tabular}{|l}
\textbf{if} $Positions$ is not empty\tabularnewline
$\quad$%
\begin{tabular}{|l}
\textbf{if} $f\in G$ \tabularnewline
$\quad$%
\begin{tabular}{|l}
pop $(E,M,B)$ from $Positions$ \tabularnewline
find rule $R$ matching cell $f$ in position $(E,M,B)$\tabularnewline
\textbf{if} $R.breaker\cap(G\setminus M)=\emptyset$ \tabularnewline
$\quad$%
\begin{tabular}{|l}
$B:=B\cup R.breaker$ \tabularnewline
$E:=E\setminus R.breaker$\tabularnewline
push $(E,M,B)$ to $Positions$\tabularnewline
\end{tabular}\tabularnewline
\end{tabular}\tabularnewline
\textbf{else}\tabularnewline
$\quad$%
\begin{tabular}{|l}
pop $(E,M,B)$ from $Positions$ \tabularnewline
\textbf{for each }rule $R$ matching cell $f$\tabularnewline
$\quad$%
\begin{tabular}{|l}
\textbf{if} $R.history\cap E=\emptyset$ and $R.breaker\cap(G\setminus M)=\emptyset$\tabularnewline
$\quad$%
\begin{tabular}{|l}
push $(E,M,B)$ to $Positions$\tabularnewline
\textbf{exit} loop\tabularnewline
\end{tabular}\tabularnewline
\end{tabular}\tabularnewline
\end{tabular}\tabularnewline
\end{tabular}\tabularnewline
\end{tabular}\tabularnewline
\begin{tabular}{c}
17.\tabularnewline
\end{tabular} & \textbf{return} $Positions$\tabularnewline
\end{tabular}

\vspace{3mm}

\dotfill
\begin{enumerate}
\item We fill the local variable $Positions$ with the unfinished position
missing the defensive breaker marks.
\item We add defensive moves to every current maker mark.
\item The position may be ruined for the maker after the first set of defensive
moves. In this case we do not need to try to add more defensive moves.
\item The defensive moves are handled differently for goal cells and history
cells that are not goal cells. First we handle the goal cells.
\item We consider the unfinished position $(E,M,B)$ and remove it from
$Positions$.
\item We find the rule that matches the current cell $f$ and the position.
The defensive breaker moves are uniquely determined by this rule since
the current cell is a goal cell and we consider every possible marking
order of the relevant history cells.
\item If any of the breaker marks determined by the matching rule are amongst
the unmarked goal cells, then the position is no longer promising
for the maker. In this case $Positions$ is left empty.
\item We update the set of breaker marks with the current defensive marks.
\item We remove the current defensive marks from the set of empty cells.
These cells are already marked by the breaker so the maker is not
able to mark them in a later turn.
\item We store the position with the new defensive moves in $Positions$.
\item Now we handle the case when the current maker mark is a history cell.
\item We consider the unfinished position $(E,M,B)$ and remove it from
$Positions$.
\item Since we do not consider the history cells for the current maker mark,
the defensive breaker marks are not uniquely determined. Hence we
need to consider every possible rule that does not contradict the
position.
\item A rule can contradict the position in two ways. A history cell of
the rule that is required to be marked by the maker cannot be in $E$
because the cells of $E$ are scheduled to be marked by the maker
at a later time. A response breaker move cannot be an unmarked goal
cell since that ruins the position for the maker.
\item We store the position in $Positions$. We do not add any defensive
breaker marks.
\item No more rules need to be considered since we already found a matching
one.
\item We return the set containing the position updated with breaker marks
or an empty set if the maker marks ruined the position for the maker. 
\end{enumerate}
\caption{AddBreaker$(G,E,M,B,F)$}
\end{algorithm}

\begin{figure}
\begin{tabular}{cccccccc}
\begin{tikzpicture} [scale=0.5]
\shadeeven{1}{1}
\shadeeven{2}{2}
\midcell{1}{1}{a}
\midcell{1}{2}{b}
\midcell{2}{2}{c}
\midcell{2}{3}{d}
\end{tikzpicture} & \begin{tikzpicture} [scale=0.5]
\shadeeven{1}{1}
\shadeeven{0}{2}
\midcell{1}{1}{a}
\midcell{1}{2}{b}
\midcell{0}{2}{c}
\midcell{0}{3}{d}
\end{tikzpicture} & \begin{tikzpicture} [scale=0.5]
\shadeeven{1}{2}
\shadeeven{2}{3}
\midcell{1}{1}{d}
\midcell{1}{2}{c}
\midcell{2}{2}{b}
\midcell{2}{3}{a}
\end{tikzpicture} & \begin{tikzpicture} [scale=0.5]
\shadeeven{1}{2}
\shadeeven{0}{3}
\midcell{1}{1}{d}
\midcell{1}{2}{c}
\midcell{0}{2}{b}
\midcell{0}{3}{a}
\end{tikzpicture} & \begin{tikzpicture} [scale=0.5]
\shadeeven{3}{1}
\shadeeven{2}{2}
\midcell{3}{1}{d}
\midcell{2}{1}{c}
\midcell{2}{2}{b}
\midcell{1}{2}{a}
\end{tikzpicture} & \begin{tikzpicture} [scale=0.5]
\shadeeven{3}{2}
\shadeeven{2}{1}
\midcell{3}{2}{d}
\midcell{2}{2}{c}
\midcell{2}{1}{b}
\midcell{1}{1}{a}
\end{tikzpicture} & \begin{tikzpicture} [scale=0.5]
\shadeeven{1}{2}
\shadeeven{2}{1}
\midcell{3}{1}{a}
\midcell{2}{1}{b}
\midcell{2}{2}{c}
\midcell{1}{2}{d}
\end{tikzpicture} & \begin{tikzpicture} [scale=0.5]
\shadeeven{1}{1}
\shadeeven{2}{2}
\midcell{3}{2}{a}
\midcell{2}{2}{b}
\midcell{2}{1}{c}
\midcell{1}{1}{d}
\end{tikzpicture}\tabularnewline
\end{tabular}

\caption{\label{fig:eightPlacements}The eight different placements of $P_{4,5}$
on a checker board.}
\end{figure}
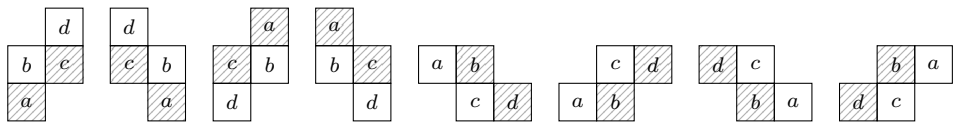

\begin{figure}
\begin{tabular}{cclccl}
\begin{tabular}{l}
1. $h_{1},h_{4},h_{5}$\tabularnewline
2. $h_{1},h_{4}h_{5}$\tabularnewline
3. $h_{1}h_{4},h_{5}$\tabularnewline
4. $h_{1},h_{5},h_{4}$\tabularnewline
5. $h_{1}h_{5},h_{4}$\tabularnewline
6. $h_{4},h_{1},h_{5}$\tabularnewline
\end{tabular} & %
\begin{tabular}{l}
7. $h_{4},h_{1}h_{5}$\tabularnewline
8. $h_{4},h_{5},h_{1}$\tabularnewline
9. $h_{4}h_{5},h_{1}$\tabularnewline
10. $h_{5},h_{1},h_{4}$\tabularnewline
11. $h_{5},h_{1}h_{4}$\tabularnewline
12. $h_{5},h_{4},h_{1}$\tabularnewline
\end{tabular} & \raise-0.9cm\hbox{
\begin{tikzpicture} [scale=0.5]
\shadeeven{1}{1}
\shadeeven{2}{2}
\shadeeven{2}{4}
\shadeeven{1}{3}
\midcell{1}{1}{a}
\midcell{1}{2}{b}
\midcell{2}{2}{c}
\midcell{2}{3}{d}
\maker{0}{1}{h_1}
\midcell{2}{1}{h_3}
\maker{3}{2}{h_5}
\maker{2}{4}{h_4}
\midcell{1}{3}{h_2}
\end{tikzpicture}
} & %
\begin{tabular}{l}
13. $a,h_{4},h_{5}$\tabularnewline
14. $a,h_{4}h_{5}$\tabularnewline
15. $ah_{4},h_{5}$\tabularnewline
16. $a,h_{5},h_{4}$\tabularnewline
17. $ah_{5},h_{4}$\tabularnewline
18. $h_{4},a,h_{5}$\tabularnewline
\end{tabular} & %
\begin{tabular}{l}
19. $h_{4},ah_{5}$\tabularnewline
20. $h_{4},h_{5},a$\tabularnewline
21. $h_{4}h_{5},a$\tabularnewline
22. $h_{5},a,h_{4}$\tabularnewline
23. $h_{5},ah_{4}$\tabularnewline
24. $h_{5},h_{4},a$\tabularnewline
\end{tabular} & \raise-0.9cm\hbox{
\begin{tikzpicture} [scale=0.5]
\shadeeven{1}{1}
\shadeeven{2}{2}
\shadeeven{2}{4}
\shadeeven{1}{3}
\maker{1}{1}{a}
\midcell{1}{2}{b}
\midcell{2}{2}{c}
\midcell{2}{3}{d}
\breaker{0}{1}{h_1}
\breaker{2}{1}{h_3}
\maker{3}{2}{h_5}
\maker{2}{4}{h_4}
\midcell{1}{3}{h_2}
\end{tikzpicture}
}\tabularnewline
\end{tabular}

\caption{\label{fig:terminal}Terminal positions.}
\end{figure}
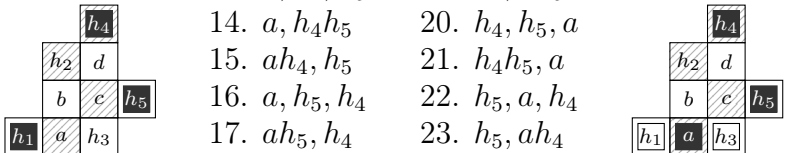

\begin{prop}
Polyomino $P_{4,5}$ is a $(2,4)$-loser and so $\tau(P_{4,5})=(1,3,11,\infty)$.\end{prop}
\begin{proof}
We used a computer program that implements Algorithms~1 and 2 to
verify that the breaker wins using the strategy of Figure~\ref{fig:historyStrat}.
The algorithm checks that the maker cannot mark all the cells of the
goal animal in any placement on the board no matter what order he
tries to mark the goal and history cells. First we find every partition
of the maker marks into classes such that every class contains one
or two cells. Then we consider each permutation of the classes. The
singleton classes represent turns where the maker used his second
mark somewhere far away on the board. The two element classes represent
turns where the maker marks two cells and both of these cells are
relevant to the position. 

During the analysis of a specific permutation of the maker marks,
we try to add the maker marks in the given order. We expect that this
process eventually fails and we are not going to be able to mark all
the goal cells. If the process succeeds, then we conclude that the
history dependent priority strategy for the breaker fails. 

We only add a maker mark on a history cell if at least one of the
corresponding defensive move sets is not in contradiction with the
position and the order of moves. We do not add any defensive moves
for a maker mark on a history cell. The missing breaker marks do not
hurt the chances of the maker. For a maker mark on a goal cell, the
defensive moves are determined since the mark order of the relevant
history cells is determined by the permutation. We add the defensive
moves for these maker marks if they do not contradict the position.
We call a position terminal if any subsequent set of maker marks cannot
be added because the corresponding breaker marks would ruin the position. 

The breaker strategy is invariant with respect to parity preserving
horizontal and vertical reflections and parity changing rotations
by 90 degrees. This implies that there is essentially one placement
of the goal animal shown in Figure~\ref{fig:placements} that we
need to consider. The labeling of the cells in Figure~\ref{fig:eightPlacements}
shows how the 8 placements are isomorphic.

Our program produces 2 different terminal positions during the search
shown in Figure~\ref{fig:terminal}. Every permutation fails after
2 or 3 turns because the breaker can spoil the position in the third
or fourth turn. So we only show the beginning turns of the permutations.
For example, in case 9 the maker tries to mark cells $h_{4}$ and
$h_{5}$ during the first turn and then cell $h_{1}$ and another
irrelevant cell during the second turn. This attempt results in a
terminal position.
\end{proof}
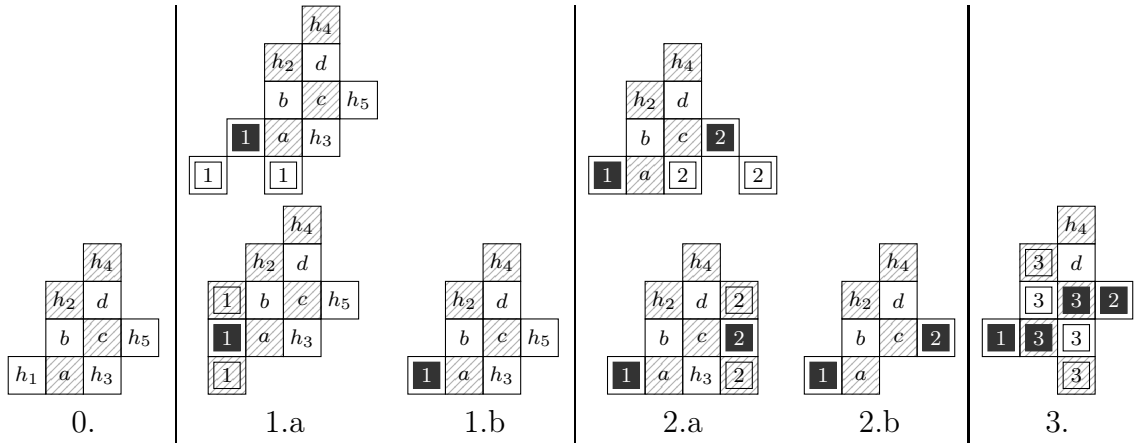
\begin{figure}
\begin{tabular}{c|cc|cc|c}
 & \begin{tikzpicture} [scale=0.5]
\shadeeven{1}{1}
\shadeeven{2}{2}
\shadeeven{2}{4}
\shadeeven{1}{3}
\midcell{1}{1}{a}
\midcell{1}{2}{b}
\midcell{2}{2}{c}
\midcell{2}{3}{d}
\maker{0}{1}{1}
\midcell{2}{1}{h_3}
\midcell{3}{2}{h_5}
\midcell{2}{4}{h_4}
\midcell{1}{3}{h_2}
\breaker{-1}{0}{1}
\breaker{1}{0}{1}
\end{tikzpicture} &  & \begin{tikzpicture} [scale=0.5]
\shadeeven{1}{1}
\shadeeven{2}{2}
\shadeeven{2}{4}
\shadeeven{1}{3}
\midcell{1}{1}{a}
\midcell{1}{2}{b}
\midcell{2}{2}{c}
\midcell{2}{3}{d}
\maker{0}{1}{1}
\breaker{2}{1}{2}
\breaker{4}{1}{2}
\maker{3}{2}{2}
\midcell{2}{4}{h_4}
\midcell{1}{3}{h_2}
\end{tikzpicture} &  & \tabularnewline
\begin{tikzpicture} [scale=0.5]
\shadeeven{1}{1}
\shadeeven{2}{2}
\shadeeven{2}{4}
\shadeeven{1}{3}
\midcell{1}{1}{a}
\midcell{1}{2}{b}
\midcell{2}{2}{c}
\midcell{2}{3}{d}
\midcell{0}{1}{h_1}
\midcell{2}{1}{h_3}
\midcell{3}{2}{h_5}
\midcell{2}{4}{h_4}
\midcell{1}{3}{h_2}
\end{tikzpicture} & \begin{tikzpicture} [scale=0.5]
\shadeeven{1}{1}
\shadeeven{2}{2}
\shadeeven{2}{4}
\shadeeven{1}{3}
\shadeeven{0}{2}
\shadeeven{0}{0}
\midcell{1}{1}{a}
\midcell{1}{2}{b}
\midcell{2}{2}{c}
\midcell{2}{3}{d}
\maker{0}{1}{1}
\midcell{2}{1}{h_3}
\midcell{3}{2}{h_5}
\midcell{2}{4}{h_4}
\midcell{1}{3}{h_2}
\breaker{0}{2}{1}
\breaker{0}{0}{1}
\end{tikzpicture} & \begin{tikzpicture} [scale=0.5]
\shadeeven{1}{1}
\shadeeven{2}{2}
\shadeeven{2}{4}
\shadeeven{1}{3}
\midcell{1}{1}{a}
\midcell{1}{2}{b}
\midcell{2}{2}{c}
\midcell{2}{3}{d}
\maker{0}{1}{1}
\midcell{2}{1}{h_3}
\midcell{3}{2}{h_5}
\midcell{2}{4}{h_4}
\midcell{1}{3}{h_2}
\end{tikzpicture} & \begin{tikzpicture} [scale=0.5]
\shadeeven{1}{1}
\shadeeven{2}{2}
\shadeeven{2}{4}
\shadeeven{1}{3}
\shadeeven{3}{3}
\shadeeven{3}{1}
\midcell{1}{1}{a}
\midcell{1}{2}{b}
\midcell{2}{2}{c}
\midcell{2}{3}{d}
\maker{0}{1}{1}
\midcell{2}{1}{h_3}
\maker{3}{2}{2}
\midcell{2}{4}{h_4}
\midcell{1}{3}{h_2}
\breaker{3}{3}{2}
\breaker{3}{1}{2}
\end{tikzpicture} & \begin{tikzpicture} [scale=0.5]
\shadeeven{1}{1}
\shadeeven{2}{2}
\shadeeven{2}{4}
\shadeeven{1}{3}
\midcell{1}{1}{a}
\midcell{1}{2}{b}
\midcell{2}{2}{c}
\midcell{2}{3}{d}
\maker{0}{1}{1}
\maker{3}{2}{2}
\midcell{2}{4}{h_4}
\midcell{1}{3}{h_2}
\end{tikzpicture} & \begin{tikzpicture} [scale=0.5]
\shadeeven{1}{1}
\shadeeven{2}{2}
\shadeeven{2}{4}
\shadeeven{1}{3}
\shadeeven{2}{0}
\maker{1}{1}{3}
\breaker{1}{2}{3}
\maker{2}{2}{3}
\midcell{2}{3}{d}
\maker{0}{1}{1}
\maker{3}{2}{2}
\midcell{2}{4}{h_4}
\breaker{1}{3}{3}
\breaker{2}{0}{3}
\breaker{2}{1}{3}
\end{tikzpicture}\tabularnewline
0. & 1.a & 1.b & 2.a & 2.b & 3.\tabularnewline
\end{tabular}

\caption{\label{fig:moveSequence}A failing attempt of the maker to use a move
sequence starting with the moves $h_{1},h_{5},ac,\ldots$ in a placement
of the goal animal.}
\end{figure}

\begin{example}
Figure~\ref{fig:moveSequence} shows why a move sequence starting
with $h_{5},h_{1},bd,\ldots$ fails to achieve $P_{4,5}$. Since $h_{1}$
is not a goal cell, we consider all possible breaker responses. Two
rules match the positions as shown in step 1.a. Since we have found
a matching rule, the analysis continues at step 1.b. We do not add
the defensive moves. Cell $h_{2}$ is again a goal cell with two rules
matching the position as shown in step 2.a. The analysis continues
at step 2.b without any of the defensive moves. Note that $h_{1}$
and $h_{5}$ are considered alone, which means the maker places the
corresponding second mark far away. The maker now tries to place cells
$a$ and $c$ in a single turn. The defensive breaker marks are now
determined because $a$ and $c$ are goal cells. The maker now marks
cell $b$ that is a priority 2 cell. This ruins the position for the
maker. 
\end{example}
More sophisticated versions of our algorithm may be needed for checking
more complicated history dependent priority strategies. One possibility
is to include the defensive breaker marks for maker marks on history
cells. These marks are not unique so we need to include all possibilities
which results in a much larger search tree. Another possibility is
to include $n$ levels of history cells. The level 1 history cells
are our usual history cells required for the goal cells. Level $k+1$
history cells are induced by the level $k$ history cells. During
the analysis, the level $n$ history cells would not produce breaker
marks, but we would include the breaker response cells for the other
history cells. Including more levels has a greater chance of success
but it is more computationally demanding.

\section{Further directions}

We list a few unanswered questions related to biased achievement games. 
\begin{enumerate}
\item Polyhex achievement games are studied in \cite{bode.harborth:hexagonal,Inagaki,sieben:hexagonal}.
What are the threshold sequences of small polyhexes?
\item Polycube achievement games are studied in \cite{higher,harary.weisbach:,sieben.deabay:polyomino}.
What are the threshold sequences of small polycubes? Adding a dimension
is a big advantage for the maker so the values in the two dimensional
threshold sequences are lower bounds for three dimensional values.
\item There are some results about achievement games played on $n$-dimensional
polycubes \cite{higher,sieben:snaky}. How does the threshold sequence
of a given polyomino change if the game is played on higher dimensional
rectangular boards?
\item Biased animal set $(1,2)$ games are studied in \cite{bode.harborth:triangle,sieben:fisher}.
In this version the maker wins if he marks any of the animals in a
given goal set of animals. What can we say about the threshold sequences
of goal sets of animals?
\item There are only finitely many $(1,1)$-winners in any animal achievement
game \cite{sieben:snaky}. Are there finitely many $(a,b)$-winners
for a fixed $a$ and $b$? The answer is most likely yes. If the answer
is in fact yes, what is the upper bound?
\item Are there two animals $A$ and $B$ with threshold sequences $(a_{1},a_{2},\ldots)$
and $(b_{1},b_{2},\ldots)$ respectively such that $a_{i}<b_{i}$
but $a_{j}>b_{j}$ for some $i$ and $j$?
\item What is the spectrum of the possible threshold sequences? For each
threshold sequence $(a_{1,}a_{2},\ldots)$ in Figures~\ref{fig:polyiamonds}
and \ref{fig:polyominoes}, $i$ divides $a_{i}+1$ for all $i$.
Is this true for all threshold sequences? One interpretation of this
property is that the $(a,\tilde{b})$ game is just as hard for the
breaker as the $(a,b)$ game if $\lfloor\tilde{b}/a\rfloor\le\lfloor b/a\rfloor$.
This seems reasonable considering Theorem~\ref{thm:main}.  
\item In a handicap $c$ game, the maker is allowed to mark $c$ cells in
the first turn and then play the usual $(1,1)$ game \cite{harary.harborth.ea:handicap}.
It is known \cite{AchievingSnaky,ito,prooftree} that Snaky is a $(1,1)$-winner
with handicap $1$. Is there a connection between a handicap $c$
game and a $(1\a c,1)$ game? Which game is easier for the maker?
\item Is there a way to use a dependency digraph to verify history dependent
priority strategies? It seems likely that the history cells should
be included in the digraph. The main difficulty is that the digraph
usually does not have any unconditional arrows, only conditional and
secondary arrows.
\item Although Snaky is conjectured to be a winner, it might actually be
a loser \cite{csernenszky,harborth.seemann:snaky*1,harborth.seemann:snaky}.
Checking priority strategies by computers is a lot easier than finding
wining strategies. So a systematic search for a history dependent
priority strategy for the breaker using a multilevel version of our
checking algorithm might not be hopeless. 
\item We could slightly improve Algorithm~2. Currently, we do not add any
defensive breaker moves for history cells. We only check (line~14)
that at least one configuration of the earlier maker marks results
in a set of breaker marks that does not ruin the position. If there
is only one such configuration, then the defensive moves could be
added together with the earlier maker marks that force this defensive
response. The addition of these marks would increase the chances of
a successful verification of the breaker strategy. It could also decrease
the branching factor of the backtracking search which would make the
search faster.
\item Polyominoes $L$, $Y$ and $Z$ are all $(1,1)$-winners but finding
a proof sequence is relatively easy for $Z$ and is quite challenging
for $L$. This intuition is strengthened by the lengths of the known
proof sequences for these animals. Can we use the threshold sequences
to firmly confirm that $Z$ is the easiest and $L$ is the hardest
five-cell animal to achieve?
\end{enumerate}
\bibliographystyle{amsplain}
\bibliography{game}

\providecommand{\bysame}{\leavevmode\hbox to3em{\hrulefill}\thinspace}
\providecommand{\MR}{\relax\ifhmode\unskip\space\fi MR }
\providecommand{\MRhref}[2]{%
  \href{http://www.ams.org/mathscinet-getitem?mr=#1}{#2}
}
\providecommand{\href}[2]{#2}
\begin{thebibliography}{10}

\bibitem{beck:book}
J{\'o}zsef Beck, \emph{Combinatorial games}, Encyclopedia of Mathematics and
  its Applications, vol. 114, Cambridge University Press, Cambridge, 2008,
  Tic-tac-toe theory.

\bibitem{ww3}
Elwyn~R. Berlekamp, John~H. Conway, and Richard~K. Guy, \emph{Winning ways for
  your mathematical plays. {V}ol. 3}, second ed., A K Peters Ltd., Natick, MA,
  2003.

\bibitem{bode.harborth:hexagonal}
Jens-P. Bode and Heiko Harborth, \emph{Hexagonal polyomino achievement},
  Discrete Math. \textbf{212} (2000), no.~1-2, 5--18, Graph theory (D\"ornfeld,
  1997).

\bibitem{bode.harborth:triangular}
\bysame, \emph{Triangular mosaic polyomino achievement}, Proceedings of the
  Thirty-first Southeastern International Conference on Combinatorics, Graph
  Theory and Computing (Boca Raton, FL, 2000), vol. 144, 2000, pp.~143--152.

\bibitem{bode.harborth:triangle}
\bysame, \emph{Triangle polyomino set achievement}, Proceedings of the
  Thirty-second Southeastern International Conference on Combinatorics, Graph
  Theory and Computing (Baton Rouge, LA, 2001), vol. 148, 2001, pp.~97--101.

\bibitem{csernenszky}
Martin Csernenszky and Andr\'as Pluh\'ar, \emph{On the complexity of
  chooser-picker games}, (preprint).

\bibitem{sieben:fisher}
Edgar Fisher and N{\'a}ndor Sieben, \emph{Rectangular polyomino set weak
  {$(1,2)$}-achievement games}, Theoret. Comput. Sci. \textbf{409} (2008),
  no.~3, 333--340.

\bibitem{sieben:sitePerimeter}
G{\'a}bor F{\"u}lep and N{\'a}ndor Sieben, \emph{Polyiamonds and polyhexes with
  minimum site-perimeter and achievement games}, Electron. J. Combin.
  \textbf{17} (2010), no.~1, Research Paper 65, 14. \MR{2644851}

\bibitem{gardner.martin:mathematical}
Martin Gardner, \emph{Mathematical games}, Sci. Amer. \textbf{240} (1979),
  18--26.

\bibitem{golomb:polyominoes}
Solomon~G. Golomb, \emph{Polyominoes: Puzzles, patterns, problem and packings},
  Princeton University Press, 1965.

\bibitem{AchievingSnaky}
Immanuel Halupczok and Jan-Christoph Schlage-Puchta, \emph{Achieving snaky},
  Integers \textbf{7} (2007), G2, 28 pp. (electronic).

\bibitem{higher}
\bysame, \emph{Some strategies for higher dimensional animal achievement
  games}, Discrete Math. \textbf{308} (2008), no.~16, 3470--3478.

\bibitem{harary:achievement*1}
Frank Harary, \emph{Achievement and avoidance games for graphs}, Graph theory
  (Cambridge, 1981), Ann. Discrete Math., vol.~13, North-Holland, Amsterdam,
  1982, pp.~111--119.

\bibitem{harary:is}
\bysame, \emph{Is {S}naky a winner?}, Geombinatorics \textbf{2} (1993), no.~4,
  79--82.

\bibitem{harary.harborth.ea:handicap}
Frank Harary, Heiko Harborth, and Markus Seemann, \emph{Handicap achievement
  for polyominoes}, Proceedings of the Thirty-first Southeastern International
  Conference on Combinatorics, Graph Theory and Computing (Boca Raton, FL,
  2000), vol. 145, 2000, pp.~65--80.

\bibitem{harary.weisbach:}
Frank Harary and M.~Weisbach, \emph{Polycube achievement games}, J.
  Recreational Math. \textbf{15} (1982--83), 241--246.

\bibitem{harborth.seemann:snaky*1}
Heiko Harborth and Markus Seemann, \emph{Snaky is an edge-to-edge looser},
  Geombinatorics \textbf{5} (1996), no.~4, 132--136.

\bibitem{harborth.seemann:snaky}
\bysame, \emph{Snaky is a paving winner}, Bull. Inst. Combin. Appl. \textbf{19}
  (1997), 71--78.

\bibitem{harborth.seemann:handicap}
\bysame, \emph{Handicap achievement for squares}, J. Combin. Math. Combin.
  Comput. \textbf{46} (2003), 47--52, 15th MCCCC (Las Vegas, NV, 2001).

\bibitem{Inagaki}
Kazumine Inagaki and Akihiro Matsuura, \emph{Winning strategies for hexagonal
  polyomino achievement}, Proceedings of the 12th WSEAS International
  Conference on Applied Mathematics (Stevens Point, Wisconsin, USA), World
  Scientific and Engineering Academy and Society (WSEAS), 2007, pp.~252--259.

\bibitem{ito}
Hiro Ito and Hiromitsu Miyagawa, \emph{Snaky is a winner with one handicap},
  8th Hellenic European Conference on Computer Mathematics and its Applications
  (2007).

\bibitem{sieben:hexagonal}
N{\'a}ndor Sieben, \emph{Hexagonal polyomino weak {$(1,2)$}-achievement games},
  Acta Cybernet. \textbf{16} (2004), no.~4, 579--585.

\bibitem{sieben:snaky}
\bysame, \emph{Snaky is a {$41$}-dimensional winner}, Integers \textbf{4}
  (2004), G5, 6 p. (electronic).

\bibitem{sieben:wild}
\bysame, \emph{Wild polyomino weak {$(1,2)$}-achievement games.},
  Geombinatorics \textbf{13(4)} (2004), 180--185.

\bibitem{sieben:polyominoes}
\bysame, \emph{Polyominoes with minimum site-perimeter and full set achievement
  games}, European Journal of Combinatorics \textbf{29} (2008), 108--117.

\bibitem{prooftree}
\bysame, \emph{Proof trees for weak achievement games}, Integers \textbf{8}
  (2008), G07, 18.

\bibitem{sieben.deabay:polyomino}
N{\'a}ndor Sieben and Elaina Deabay, \emph{Polyomino weak achievement games on
  {$3$}-dimensional rectangular boards}, Discrete Mathematics \textbf{290}
  (2005), 61--78.

\end{thebibliography}

\end{document}